\documentclass{amsart}

\usepackage{amsthm}
\usepackage{amssymb}
\usepackage{amsmath}
\usepackage{amsfonts}
\usepackage{amsrefs}
\usepackage{textcomp}
\usepackage[T1]{fontenc}
\usepackage[utf8x]{inputenc}
\usepackage{textcomp}
\usepackage{wasysym}
\usepackage{stmaryrd}
\usepackage{esint}
\usepackage[all]{xy}
\usepackage{graphicx}
\usepackage{bbm}
\usepackage{xcolor}

\newtheorem{theorem}{Theorem}[section]
\newtheorem{proposition}[theorem]{Proposition}
\newtheorem{lemma}[theorem]{Lemma}
\newtheorem{corollary}[theorem]{Corollary}
\newtheorem{example}[theorem]{Example}

\theoremstyle{definition}
\newtheorem{definition}[theorem]{Definition}
\newtheorem{remark}[theorem]{Remark}

\newcommand{\res}{\mathop{\hbox{\vrule height 7pt width .5pt depth 0pt \vrule height .5pt width 6pt depth 0pt}}\nolimits}

\DeclareMathOperator{\spt}{spt}
\DeclareMathOperator{\Lip}{Lip}

\usepackage{pgf,tikz}
\usetikzlibrary{arrows}
\definecolor{zzttqq}{rgb}{0.6,0.2,0}
\definecolor{uququq}{rgb}{0.25,0.25,0.25}
\definecolor{qqqqff}{rgb}{0,0,1}
\definecolor{xdxdff}{rgb}{0.49,0.49,1}
\usepackage{pgf,tikz}
\usetikzlibrary{arrows}
\definecolor{ffffqq}{rgb}{1,1,0}
\definecolor{fftttt}{rgb}{1,0.2,0.2}
\definecolor{ffqqqq}{rgb}{1,0,0}
\definecolor{ffqqtt}{rgb}{1,0,0.2}
\definecolor{ffttqq}{rgb}{1,0.2,0}
\definecolor{xdxdff}{rgb}{0.49,0.49,1}
\definecolor{qqqqff}{rgb}{0,0,1}
\definecolor{ffttww}{rgb}{1,0.2,0.4}
\definecolor{ttqqcc}{rgb}{0.2,0,0.8}
\definecolor{ffqqww}{rgb}{1,0,0.4}
\definecolor{xdxdff}{rgb}{0.49,0.49,1}
\definecolor{qqqqff}{rgb}{0,0,1}
\definecolor{uququq}{rgb}{0.25,0.25,0.25}
\definecolor{ttffww}{rgb}{0.2,1,0.4}
\definecolor{ffqqqq}{rgb}{1,0,0}
\definecolor{qqfftt}{rgb}{0,1,0.2}
\definecolor{zzttqq}{rgb}{0.6,0.2,0}
\definecolor{ffffqq}{rgb}{1,1,0}
\definecolor{qqffww}{rgb}{0,1,0.4}
\definecolor{xdxdff}{rgb}{0.49,0.49,1}
\definecolor{qqqqff}{rgb}{0,0,1}

\definecolor{ffqqqq}{rgb}{1,0,0}
\definecolor{ttqqff}{rgb}{0.2,0,1}
\definecolor{ttqqcc}{rgb}{0.2,0,0.8}
\definecolor{uququq}{rgb}{0.25,0.25,0.25}
\definecolor{xdxdff}{rgb}{0.49,0.49,1}
\definecolor{qqqqff}{rgb}{0,0,1}
\definecolor{ffqqtt}{rgb}{1,0,0.2}

\definecolor{qqffqq}{rgb}{0,1,0}
\definecolor{qqfftt}{rgb}{0,1,0.2}
\definecolor{ffffqq}{rgb}{1,1,0}
\definecolor{ffttcc}{rgb}{1,0.2,0.8}
\definecolor{ffwwqq}{rgb}{1,0.4,0}
\definecolor{ffqqqq}{rgb}{1,0,0}
\definecolor{ffffqq}{rgb}{1,1,0}
\definecolor{qqfftt}{rgb}{0,1,0.2}
\definecolor{ffttcc}{rgb}{1,0.2,0.8}
\definecolor{ffwwqq}{rgb}{1,0.4,0}
\definecolor{ffqqqq}{rgb}{1,0,0}
\usepackage{subfig}
\usetikzlibrary{arrows}
\definecolor{ffqqtt}{rgb}{1,0,0.2}
\definecolor{ffffqq}{rgb}{1,1,0}
\definecolor{ttffqq}{rgb}{0.2,1,0}
\definecolor{wwwwww}{rgb}{0.4,0.4,0.4}
\definecolor{zzqqzz}{rgb}{0.6,0,0,6}
\usepackage{mathtools}

\usepackage{pgf,tikz}
\usetikzlibrary{arrows}
\definecolor{ffffqq}{rgb}{1,1,0}
\definecolor{qqffqq}{rgb}{0,1,0}
\definecolor{ttttff}{rgb}{0.2,0.2,1}
\definecolor{ffqqtt}{rgb}{1,0,0.2}
\usepackage[active]{srcltx}
\usepackage{amsmath,bbm}
\usepackage{amsfonts}
\usepackage{pgf,tikz}
\usepackage{amsmath,accents,eqnarray}
\usetikzlibrary{arrows}
\usepackage{amssymb}
\usepackage[colorlinks=true,urlcolor=blue,
citecolor=red,linkcolor=blue,linktocpage,pdfpagelabels,bookmarksnumbered,bookmarksopen]{hyperref}
\usepackage{graphicx, tikz}
\usepackage{pgf,tikz}
\usetikzlibrary{arrows}
\definecolor{wwqqww}{rgb}{0.4,0,0,4}
\definecolor{wwffqq}{rgb}{0.4,1,0}
\definecolor{fffftt}{rgb}{1,1,0.2}
\definecolor{ccqqcc}{rgb}{0.8,0,0.8}
\definecolor{ffffww}{rgb}{1,1,0.4}
\definecolor{ttfftt}{rgb}{0.2,1,0.2}
\definecolor{fffftt}{rgb}{1,1,0.2}
\definecolor{ccqqww}{rgb}{0.8,0,0.4}
\definecolor{ffffww}{rgb}{1,1,0.4}
\definecolor{ttfftt}{rgb}{0.2,1,0.2}
\definecolor{uququq}{rgb}{0.25,0.25,0.25}
\definecolor{ffqqtt}{rgb}{1,0,0.2}
\definecolor{ttffff}{rgb}{0.2,1,1}
\definecolor{fffftt}{rgb}{1,1,0.2}
\definecolor{wwffqq}{rgb}{0.4,1,0}
\definecolor{qqqqff}{rgb}{0,0,1}
\definecolor{uququq}{rgb}{0.25,0.25,0.25}
\definecolor{ffqqtt}{rgb}{1,0,0.2}
\definecolor{fffftt}{rgb}{1,1,0.2}
\definecolor{wwffqq}{rgb}{0.4,1,0}
\definecolor{qqqqff}{rgb}{0,0,1}
\definecolor{ttfftt}{rgb}{0.2,1,0.2}
\definecolor{qqqqff}{rgb}{0,0,1}
\definecolor{ffqqtt}{rgb}{1,0,0.2}
\definecolor{fffftt}{rgb}{1,1,0.2}
\definecolor{ffffqq}{rgb}{1,1,0}
\definecolor{qqfftt}{rgb}{0,1,0.2}
\tikzset { domaine/.style 2 args={domain=#1:#2} }
\tikzset{
xmin/.store in=\xmin, xmin/.default=-3, xmin=-3,
xmax/.store in=\xmax, xmax/.default=3, xmax=3,
ymin/.store in=\ymin, ymin/.default=-3, ymin=-3,
ymax/.store in=\ymax, ymax/.default=3, ymax=3,
}

\usepackage{pgf,tikz}
\usetikzlibrary{arrows}
\definecolor{ffffqq}{rgb}{1,1,0}
\definecolor{ttfftt}{rgb}{0.2,1,0.2}
\definecolor{qqffcc}{rgb}{0,1,0.8}
\definecolor{ffwwqq}{rgb}{1,0.4,0}
\definecolor{ffqqqq}{rgb}{1,0,0}
\definecolor{ttffqq}{rgb}{0.2,1,0}
\definecolor{qqqqff}{rgb}{0,0,1}
\definecolor{ffqqtt}{rgb}{1,0,0.2}
\definecolor{ffffqq}{rgb}{1,1,0}
\definecolor{qqfftt}{rgb}{0,1,0.2}
\usepackage[english]{babel}
\usepackage{marginnote}
\usepackage[left=2.95cm,right=2.95cm,top=2.8cm,bottom=2.8cm]{geometry}
\numberwithin{equation}{section}
\usepackage{mathrsfs}
\usepackage{color}
\usepackage{amsmath, bbm}
\usepackage{amsfonts}
\usepackage{amssymb}
\usepackage{graphicx}%
\providecommand{\U}[1]{\protect\rule{.1in}{.1in}}
\definecolor{linkcolor}{rgb}{0.00,0.50,0.00}
\providecommand{\U}[1]{\protect\rule{.1in}{.1in}}
\definecolor{ffffqq}{rgb}{1,1,0}
\definecolor{qqffqq}{rgb}{0,1,0}
\definecolor{ttttff}{rgb}{0.2,0.2,1}
\definecolor{ffqqtt}{rgb}{1,0,0.2}

\def\XXint#1#2#3{{\setbox0=\hbox{$#1{#2#3}{\int}$}
     \vcenter{\hbox{$#2#3$}}\kern-.5\wd0}}

\newcommand{\twopartdef}[4]
{
\left\{
		\begin{array}{ll}
			#1 & #2 \\
			#3 & #4
		\end{array}
	\right.
}

\newcommand{\threepartdef}[6]
{
	\left\{
		\begin{array}{lll}
			#1 & \mbox{if } #2 \\
			#3 & \mbox{if } #4 \\
			#5 & \mbox{if } #6
		\end{array}
	\right.
}

\usepackage{geometry}
\author{Samer Dweik and Wojciech G\'{o}rny}

\address{S. Dweik: Department of Mathematics, University of British Columbia, Vancouver, BC Canada V6T 1Z2.\\
W. G\'{o}rny: Faculty of Mathematics, Informatics and Mechanics, University of Warsaw, Warsaw, Poland.}

\email{dweik@math.ubc.ca, w.gorny@mimuw.edu.pl}

\date{\today}


\title{Least gradient problem on annuli}


\begin{document}

\begin{abstract}
We consider the two dimensional BV least gradient problem on an annulus with given boundary data $g \in BV(\partial\Omega)$. Firstly, we prove that this problem is equivalent to the optimal transport problem with source and target measures located on the boundary of the domain. Then, under some admissibility conditions on the trace, we show that there exists a unique solution for the BV least gradient problem. Moreover, we prove some $L^p$ estimates on the corresponding minimal flow of the Beckmann problem, which implies directly $W^{1,p}$ regularity for the solution of the BV least gradient problem.
\end{abstract}
\maketitle
\section{Introduction}
 
In this paper, we are interested in the study of the planar least gradient problem (see, for instance, \cite{Dweik,Gorny0,Gorny,Sternberg}):
\begin{equation}\label{BV least gradient problem classical}
\min\bigg\{\int_{\Omega} |D u|\,:\,u \in BV(\Omega),\,T u=g\bigg\},
\end{equation}
where $\Omega \subset \mathbb{R}^2$ is an open bounded set with Lipschitz boundary and $T: BV(\Omega) \rightarrow L^1(\partial\Omega)$ denotes the trace operator. This problem is typically considered under the assumption of strong convexity of $\Omega$; in this paper, we aim to relax this assumption and provide an analysis of the least gradient problem on an annulus.

In \cite{Sternberg}, the authors prove existence and uniqueness of the solution $u$ for the least gradient problem \eqref{BV least gradient problem classical} in the case where the domain $\Omega$ is strictly convex and the boundary datum $g \in  C(\partial\Omega)$. In addition, the case where $f \in L^1(\partial\Omega)$ is also studied in \cite{Mazon}, where the authors prove existence of a solution for the following relaxation of \eqref{BV least gradient problem classical}:
\begin{equation}\label{BV least gradient problem classical relaxation}
\min\bigg\{\int_{\Omega} |D u| + \int_{\partial\Omega} |T u - g|\,\mathrm{d}\mathcal{H}^1(x)\,:\,u \in BV(\Omega)\bigg\}.
\end{equation}\\
We note that if a solution $u$ of \eqref{BV least gradient problem classical relaxation} satisfies $T u =g$, then $u$ is clearly a solution of \eqref{BV least gradient problem classical}. Moreover, the authors of \cite{Spradlin} provide an example of boundary data $g \in L^\infty(\partial\Omega)$ such that a solution $u$ for \eqref{BV least gradient problem classical} does not exist. On the other hand, the problem \eqref{BV least gradient problem classical} has a solution as soon as $g \in BV(\partial\Omega)$ and the domain $\Omega$ is strictly convex (see, for instance, \cite{Gorny,Dweik,Moradifam}). However, if we relax the assumption of continuity of boundary data, we lose uniqueness of minimizers.  

The first attempt to prove existence of minimizers when the domain is not strictly convex has been made in \cite{Sabra}, where the authors considered a domain $\Omega$ which is convex, but not strictly convex. Unfortunately, in this case we cannot expect, in general, existence of a solution for \eqref{BV least gradient problem classical}, even if the boundary datum $g$ is smooth. To see that, let us consider a square $\Omega:=\left[0,1\right]^2$ and take $h \in C^\infty_c((0,1))$. We define $g(x_1,1)=h(x_1)$, for all $x_1 \in [0,1]$, and $g(x_1,x_2)=0$, if $x_2 <1$. We see that the level sets of a solution to \eqref{BV least gradient problem classical relaxation} are contained in the segment $[0,1] \times \{1\}$, hence it does not satisfy $T u =g$. This led the authors of \cite{Sabra} to prove existence and uniqueness of solutions to the problem \eqref{BV least gradient problem classical} under some admissibility conditions on the behavior of boundary data on the flat parts of $\partial\Omega$. In this paper, we will follow a similar approach and provide a set of admissibility conditions under which we will prove existence and uniqueness of solutions. 

We will approach this problem using a link between the least gradient problem and the optimal transport problem. For a convex domain $\Omega$, the authors of \cite{Dweik,Gorny0} prove that the problem \eqref{BV least gradient problem classical} is equivalent to the Beckmann problem \cite{Beckmann} with source and target measures located on the boundary $\partial\Omega$, which is in turn related to the optimal transport problem with Euclidean cost. In other words, the problem \eqref{BV least gradient problem classical} is equivalent to:
\begin{equation} 
 \label{BeckmannIntro}\min\bigg\{\int_{\bar{\Omega}} |v|\,:\,v \in \mathcal{M}^d(\bar{\Omega}),\,\nabla \cdot v=0\,\,\,\,\mbox{and}\,\,\,v \cdot \nu =f\,\,\mbox{on}\,\,\,\partial\Omega\bigg\}.
\end{equation}
The equivalence between the least gradient problem \eqref{BV least gradient problem classical} and the Beckmann problem \eqref{BeckmannIntro} follows from the fact that if $u \in BV(\Omega)$ with $T u = g$, then we see easily that $v:=R_{\frac{\pi}{2}} D u$ is an admissible flow in \eqref{BeckmannIntro} with $f=\partial_{\tau}g$, where $\partial_\tau g$ denotes the tangential derivative of $g$. On the other hand, given a flow $v$ such that $\nabla \cdot v = 0$ and $v \cdot \nu = f$ on $\partial\Omega$, there is a function $u$ such that $v=R_{\frac{\pi}{2}} D u$. Furthermore, if $|v|$ gives zero mass to the boundary (i.e. $|v|(\partial\Omega)=0$), then $T u = g$. In other words, there is a one-to-one correspondence between vector measures $Du$ in \eqref{BV least gradient problem classical relaxation}
(considered as measures on $\bar{\Omega}$, so that we also include the part of the derivative of $u$ which is on the boundary, i.e. the possible jump from $u_{|\partial\Omega}$ to $g$) and vector measures $v$ in \eqref{BeckmannIntro}. In particular, this implies that if $v$ is an optimal flow for the Beckmann problem \eqref{BeckmannIntro} such that $|v|$ gives zero mass to the boundary, then a solution $u$ for the problem \eqref{BV least gradient problem classical relaxation} turns out to be a solution for \eqref{BV least gradient problem classical}.

In addition, it is well known that the Beckmann problem \eqref{BeckmannIntro} is completely equivalent to the Monge-Kantorovich \cite{Kanto,Monge} optimal transportation problem (see, for instance, \cite{San2}): 
\begin{equation}\label{KantoIntro}\min\bigg\{\int_{\bar{\Omega} \times \bar{\Omega}} |x-y|\,\mathrm{d}\gamma\,:\,\gamma \in  \mathcal{M}^+(\bar{\Omega} \times \bar{\Omega}),\,(\Pi_x)_{\#}\gamma=f^{+} \,\,\mbox{and}\,\,(\Pi_y)_{\#} \gamma =f^{-}\bigg\},
\end{equation}\\
where $f^+$ and $f^-$ are the positive and negative parts of $f$. Moreover, the dual of \eqref{KantoIntro} is the following:
\begin{equation}\label{Kanto dualIntro} 
    \sup\bigg\{\int_{\bar{\Omega}} \phi\,\mathrm{d}(f^+ - f^-)\,:\,\phi \in  \Lip_1(\Omega)\bigg\}.
\end{equation}\\
From \cite{San2}, we have that every optimal flow $v$ for \eqref{BeckmannIntro} is of the form $v=-|v|\nabla \phi$, where $\phi$ is the Kantorovich potential (i.e. a maximizer of \eqref{Kanto dualIntro}). Moreover, the solution $v$ is unique as soon as at least one between $f^+$ and $f^-$ is in $L^1(\Omega)$ and $|v| \in L^p(\Omega)$ provided that $f^\pm \in L^p(\Omega)$ (for every $p \in [1,\infty]$); unfortunately, this is not the case here since our measures $f^+$ and $f^-$ are concentrated on $\partial\Omega$ and we need to prove this in our setting. Anyway, we note that as soon as we prove uniqueness of the optimal flow $v$ for \eqref{BeckmannIntro}, then we get directly uniqueness of the solution $u$ (if it exists) for \eqref{BV least gradient problem classical}. In addition, the $L^p$ summability of the minimal flow $v$ implies eventually a $W^{1,p}$ regularity for the solution $u$ of the BV least gradient problem \eqref{BV least gradient problem classical}. In \cite{Dweik}, the authors prove that the problem \eqref{BeckmannIntro} has a unique minimizer $v$ as soon as $f^+$ or $f^-$ is atomless and $\Omega$ is strictly convex. If both $f^+,\,f^-$ are in $L^p(\Omega)$ with $p \leq 2$, then $|v| \in L^p(\Omega)$. In other words, the BV least gradient problem \eqref{BV least gradient problem classical} reaches a minimum as soon as \,$\Omega$\, is strictly convex and, the solution \,$u$\, of \eqref{BV least gradient problem classical} is unique provided that $g \in C(\partial\Omega)$. And, the solution $u$ is in $W^{1,p}(\Omega)$ as soon as $g \in W^{1,p}(\partial\Omega)$ and $p \leq 2$. Moreover, the authors of \cite{Dweik} give a counter-example to the $W^{1,p}$ regularity of $u$ for $p>2$. In addition, we note that there are some $C^{0,\alpha}$ results about the solution $u$ of the problem \eqref{BV least gradient problem classical} (see, for instance, \cite{Dweik,GornyNonstconv,Sternberg}). More precisely, we have $g\in C^{0,\alpha}(\partial\Omega) \Rightarrow u \in C^{0,\frac{\alpha}{2}}(\Omega)$ and $g\in C^{1,\alpha}(\partial\Omega) \Rightarrow u \in C^{0,\frac{\alpha +1}{2}}(\Omega)$.\\

In this paper, we consider the least gradient problem \eqref{BV least gradient problem classical} on an annulus, so the domain $\Omega$ is not convex; even its boundary is not connected. To be more precise, let $\Omega_\pm$ be two bounded strictly convex domains such that $\Omega_- \subset \subset \Omega_+$. Then, we consider the planar least gradient problem on an annulus $\Omega = \Omega_+ \backslash \overline{\Omega_-}$. As the annulus is not strictly convex (even convex), we do not have any general results concerning existence of minimizers for \eqref{BV least gradient problem classical}. Again, we may point out very easy boundary data such that the corresponding least gradient problem \eqref{BV least gradient problem classical} has no minimizer - suppose that $g|_{\partial\Omega_+} \equiv 1$ and $g|_{\partial\Omega_-} \equiv 0$. However, the tangent derivative of $g$ equals zero and there exists a (zero) solution of the Beckmann problem. Hence, we will not consider the least gradient problem, but rather prove equivalence between the Beckmann problem \eqref{BeckmannIntro} and the following problem
\begin{equation}\label{BV least gradient problem on annulus}
\min\bigg\{\int_{\Omega} |Du|\,:\,u \in BV(\Omega), \, \partial_\tau(Tu) = f \bigg\}
\end{equation}\\
and then pass from this problem to the usual least gradient problem \eqref{BV least gradient problem classical} for an admissible function $\widetilde{g} \in BV(\partial\Omega)$ such that $\partial_\tau \widetilde{g} = f$. In other words, we allow for certain vertical shifts of the values of $g$ on each of the connected components of $\partial\Omega$. As the fundamental group of an annulus is nontrivial, it is not obvious that from a divergence-free vector field $v$ we may recover a function $u$ such that $v = R_{\frac{\pi}{2} Du}$ and the method introduced in \cite{Gorny0} applies; another difficulty is that if the domain $\Omega$ is not convex, it is not obvious if we have the equivalence between \eqref{BeckmannIntro} and \eqref{KantoIntro}  (see \cite{Dweik,Gorny0}). We deal with these issues in Section \ref{Sec. 3}.

As the non-connectedness of the boundary plays a role, we do not aim to prove a general result concerning existence and uniqueness of minimizers, but rather a set of quite general sufficient conditions that imply existence of a solution for \eqref{BV least gradient problem on annulus} (it may be hard to find a set of conditions which is both necessary and sufficient to get existence of minimizers for \eqref{BV least gradient problem on annulus}; see also the work of Rybka and Sabra \cite{Sabra} which concerns the case where the domain is convex). Then, under the same structural hypotheses, we pass from a solution to problem \eqref{BV least gradient problem on annulus} to a solution of the usual BV least gradient problem \eqref{BV least gradient problem classical}. We will address these issues in Section \ref{Sec. 4}.

Another problem is the regularity of least gradient functions. As we are using techniques derived from optimal transport, we want to extend the results proved in \cite{Dweik} concerning $W^{1,p}$ regularity of least gradient functions. However, these require uniform convexity of $\Omega$. Under the structural hypotheses introduced in Section \ref{Sec. 4}, we work around this problem and prove $L^p$ summability of the transport density, which translates to $W^{1,p}$ regularity of solutions to the least gradient problem. We will address these issues in Section \ref{Sec. 5}.

Finally, in Section \ref{Sec. 6}, we discuss the limits and possible extensions to the approach presented in this paper. We focus on two issues: the first one is the optimality of our structural assumptions and possible extensions to general Lipschitz domains; the second one is validity of our results for strictly convex norms on $\mathbb{R}^2$ other than the Euclidean norm.

\section{Preliminaries}\label{Sec. 2}

This Section serves two purposes. In the first part, we recall basic properties of least gradient functions. In the second part, we study what is the structure of least gradient functions on an annulus $\Omega$. Here and in the whole paper, we introduce the following notation.

\begin{definition}\label{dfn:annulus}
We say that $\Omega \subset \mathbb{R}^2$ is an annulus, if $\Omega = \Omega_+ \backslash \overline{\Omega}_-$, where $\Omega_\pm$ are open bounded strictly convex subsets of $\mathbb{R}^2$ such that $\Omega_- \subset \subset \Omega_+$. Let $g \in L^1(\partial\Omega)$. Then, we will denote by $g_{\pm}$ the restrictions $g_\pm = g_{|\partial\Omega_\pm} \in L^1(\partial\Omega_\pm)$ and denote by $T_\pm$ the trace operator $T: BV(\Omega) \rightarrow L^1(\partial\Omega) = L^1(\partial\Omega_+) \oplus L^1(\partial\Omega_-)$ composed with a projection onto $L^1(\partial\Omega_\pm)$. 

Moreover, if $g_\pm \in BV(\partial\Omega_\pm)$, we will denote by $f_\pm = \partial_\tau g_\pm \in \mathcal{M}(\partial\Omega_\pm)$ its tangential derivative and decompose it into a positive part $f_\pm^+$ and negative part $f_\pm^-$. 
\end{definition}

This paper is devoted to the study of least gradient functions on annuli, nevertheless in the following results we will clearly state if they are valid only for annuli, only for Lipschitz domains, or for general open sets.

\subsection{Least gradient functions}
 In this subsection, we recall the definition and some properties of least gradient functions (see also \cite{Bombieri,Giusti}). Then, we prove some results concerning pointwise properties of precise representatives of least gradient functions.

\begin{definition}
We say that $u \in BV(\Omega)$ is a function of least gradient if for every compactly supported $\phi \in BV(\Omega)$, we have 
\begin{equation*}
\int_\Omega |D u| \leq \int_\Omega |D(u + \phi)|.
\end{equation*}
Let us note that due to \cite[Theorem 2.2]{Ziemer}, we may equivalently assume that $\phi$ has trace zero. We also say that $u$ is a solution of the least gradient problem for $g \in L^1(\partial\Omega)$ in the sense of traces, if $u$ is a least gradient function such that $Tu = g$.
\end{definition}

To deal with regularity of least gradient functions, it is convenient to consider superlevel sets of $u$, i.e. sets of the form $\partial \{ u > t \}$ for $t \in \mathbb{R}$. A classical theorem states that

\begin{theorem}\label{twierdzeniezbgg}
$ ($\cite[Theorem 1]{Bombieri}$) $ \\
Suppose $\Omega \subset \mathbb{R}^d$ is open. Let $u$ be a function of least gradient in $\Omega$. Then, the set $\partial \{ u > t \}$ is minimal in $\Omega$, i.e. $\chi_{\{ u > t \}}$ is of least gradient, for every $t \in \mathbb{R}$. 
\end{theorem}
Obviously, Theorem \ref{twierdzeniezbgg} also holds for sets of the form $\{ u \geq t \}$. Let us introduce a convention in which we identify a set of finite perimeter with the set of its points of positive density. Under this convention, in dimension two (see, for instance, \cite[Chapter 10]{Giusti}) the boundary $\partial E$ of a minimal set $E$ is a locally finite union of line segments. In particular, if we take the precise representative of a least gradient function $u$, then $\partial \{ u \geq t \}$ is a locally finite union of line segments for every $t$. For this reason, we will in this paper always assume that $u$ is the precise representative of a least gradient function in order to be able to state any pointwise results.

\subsection{Traces of least gradient functions on annuli}

In \cite{Sternberg}, the authors have shown existence and uniqueness of solutions to the least gradient problem for continuous boundary data and strictly convex $\Omega$ (or, to be more precise, the authors assume that $\partial \Omega$ has non-negative mean curvature and is not locally area-minimizing; in dimension two, these conditions are equivalent to strict convexity). The proof of existence is constructive and its main idea is reversing Theorem \ref{twierdzeniezbgg} in order to construct almost all level sets of the solution. However, the authors provide counterexamples if the domain fails to be strictly convex.

In this subsection, we look at least gradient functions defined on annuli. We are particularly interested in their traces - as the domain is not strictly convex, not all continuous traces will be admissible. In particular, we will see why restriction of boundary data to the class $BV(\partial\Omega)$ in our analysis is reasonable. Moreover, the results presented in this subsection are of independent interest as interior regularity results for least gradient functions on strictly convex domains.


On an annulus $\Omega$, Theorem \ref{twierdzeniezbgg} gives us an important restriction on the shape of superlevel sets $E_t$. As connected components of $\partial E_t$ are line segments which lie entirely inside $\Omega$, each of these line segments which starts at a point of $\partial\Omega_-$ has to end at a point from $\partial\Omega_+$; however, the converse is not necessarily true. In view of Theorem \ref{thm:equivalence} (see below), for boundary data $g \in BV(\partial\Omega)$, we may think of each connected component of $\partial E_t$ as a transport ray in a corresponding transport problem. In this formulation, this observation means that there is no transport between points of $\partial\Omega_-$, but there may be transport between points of $\partial\Omega_+$. 



We start with proving that the total variation of $g$ restricted to $\partial\Omega_-$ is finite. Then, we will show that $g$ posseses some additional structure resulting from the topology of $\Omega$. Apart from their value as regularity results for least gradient functions on annuli, they serve as a justification for the choice of assumptions under which we prove existence of minimizers in Section \ref{Sec. 4}.

\begin{lemma}\label{lem:innerbv}
Let $\Omega \subset \mathbb{R}^2$ be an annulus. Suppose that $u \in BV(\Omega)$ is a least gradient function with trace $g \in L^\infty(\partial\Omega)$. Then $g_{|\partial\Omega_-} \in BV(\partial\Omega_-)$.
\end{lemma} 

\begin{proof}
Let us denote by $P(E,U)$ the perimeter of a set $E$ with respect to an open set $U$. We begin by noticing that all minimal sets in $\Omega$ have perimeter less or equal to $P(\Omega, \mathbb{R}^2)$ (see, for instance, \cite[Lemma 2.17]{GornyGen}). By Theorem \ref{twierdzeniezbgg}, $\{ u \geq t \}$ is a minimal set for every $t$, i.e. its characteristic function is a function of least gradient; furthermore, for almost all $t \in \mathbb{R}$, the trace of $\chi_{\{u \geq t \}}$ equals $\chi_{\{ g \geq t \}}$. From now on, we consider only such $t$. 

As $\Omega$ is a convex subset of the plane, $\partial\Omega$ is homeomorphic to a circle; we consider the one-dimensional BV space on $\partial\Omega$. For the equivalence between the one-dimensional definitions of BV spaces on lines, see for instance \cite{EvansGariepy}; this equivalence extends to one-dimensional boundaries, see for instance \cite{Gorny}. By the co-area formula for $g_{|\partial\Omega_-}$, we have
$$ |D g|(\partial\Omega_-) = \int_{\mathbb{R}} P(\{ g  \geq t \}, \partial\Omega_-) dt.$$
Suppose that $g_{|\partial\Omega_-} \notin BV(\partial\Omega_-)$. Then the left hand side is infinite. Hence, the integrand on the right hand side is unbounded and, for any $M > 0$, we can find $t \in \mathbb{R}$ so that $P(\{ g \geq t \}, \partial\Omega_-) \geq M$. As $\partial\Omega_-$ is  one-dimensional, if $P(\{ g \geq t \}, \partial\Omega_-)$ is finite, it is a natural number (for the characterization of the BV space in one dimension, see for instance \cite[Chapter 5.10]{EvansGariepy}). Take a minimal set $E_t$ with trace $\chi_{\{ g \geq t \}}$; then, at each of the points from $\partial^* \{ g \geq t \}$, the reduced boundary of $\{ g \geq t \}$, there is a line segment from $\partial^* \{ u \geq t \}$ which ends at this point. However, no such line segment may connect two points from $\partial\Omega_-$; hence each of these line segments goes from $\partial\Omega_-$ to $\partial\Omega_+$. Then,
$$ P(\Omega, \mathbb{R}^2) \geq M \text{ dist}(\partial\Omega_-, \partial\Omega_+).$$
However, $M$ was arbitrary and $P(E_t, \Omega)$ is bounded, which yields to a contradiction. Hence $g_{|\partial\Omega_-} \in BV(\partial\Omega_-)$.  $\qedhere$
\end{proof}

However, the structure of \,$\Omega$\, imposes even stricter conditions on the structure of $g_{|\partial\Omega_-}$. The following results serve as motivations for admissibility conditions (H1)-(H4) in Section \ref{Sec. 4}; they do not enter the proof 
of equivalence between the least gradient problem and the optimal transport one
(see Section \ref{Sec. 3}) and so, we will use this equivalence to prove them.\\

First, as no line segment $l \subset \partial E_t$ may have both ends on $\partial\Omega_-$, the total variation of $g$ on $\partial\Omega_-$ is smaller than the total variation of $g$ on $\partial\Omega_+$.

\begin{lemma}\label{lem:tvinequality}
Let $\Omega \subset \mathbb{R}^2$ be an annulus. Suppose that $u \in BV(\Omega)$ is a least gradient function with trace $g \in BV(\partial\Omega)$. Then $TV(g_-) \leq TV(g_+)$.
\end{lemma} 

\begin{proof}
Set $f=\partial_\tau g$. We will use Proposition \ref{prop:admissibility}, which was proved as a step in the proof of \cite[Theorem 2.1]{Gorny0}. It states that a rotation of the gradient of a $BV$ function is an admissible vector field in \eqref{BeckmannIntro}, i.e. the Beckmann problem. In particular, boundaries of superlevel sets correspond to transport rays.

Divide the derivative $f$ into four parts: $f_-^+$, $f_-^-$, $f_+^+$ and $f_+^-$. As $u$ is a least gradient function, there is no boundary of a superlevel set which connects two points from $\partial\Omega_-$. Hence, there can be no transport from $\partial\Omega_-$ to $\partial\Omega_-$, so $f_-^+$ is transported to $f_+^-$
; similarly, a part of \,$f_+^+$\, is transported to $f_-^-$.   
Summing up these inequalities, we obtain $TV(g_-) \leq  TV(g_+)$. $\qedhere$
\end{proof}

Moreover, we have the following:
\begin{proposition}\label{stw:specialconfiguration}
Let \,$\Omega \subset \mathbb{R}^2$ be an annulus. Suppose that $u \in BV(\Omega)$ is a least gradient function with trace \,$g \in BV(\partial\Omega)$ and set \,$f=\partial_\tau g$. Then, we have \,$\spt(f_-^+) \cap \spt(f_-^-) = \{ p_1, ..., p_k \}$. Moreover, for each \,$i = 1, ..., k$, the point $p_i$ lies on a line segment \,$l_i \subset \partial E_t$ with both ends on $\partial\Omega_+$.
\end{proposition} 

\begin{proof}
Suppose that $p \in \spt(f_-^+) \cap \spt(f_-^-)$. For every $n$, consider the sets $V_n = \partial\Omega_- \cap B(p, \frac{1}{n})$. Inside any $V_n$, pick two points $p_n^+,\,p_n^-$ such that $p_n^+ \in \spt(f_-^+)$ and $p_n^- \in \spt(f_-^-)$. So, there are two corresponding points $q_n^\pm \in \partial\Omega_+$ such that $\left[q_n^+,p_n^-\right]$ and $\left[p_n^+,q_n^-\right]$ are two transport rays.
From the cyclical monotonicity property of the optimal transport plan for \eqref{KantoIntro} (see, for instance, \cite[Chapter 1]{San2} or Lemma \ref{transport rays inside the annulus}), we have
$$ |p_n^+ - q_n^-| + |p_n^- - q_n^+| \leq |p_n^+ - p_n^-| + |q_n^+ - q_n^-|.$$
Yet, up to a subsequence, we have $q_n^\pm \rightarrow q^\pm$, where $q^\pm \in \partial\Omega_+$. Then, passing to the limit when $n \rightarrow \infty$, we get
$$ |p - q^+| + |p -q^-| \leq |q^+- q^-|,$$
which implies that $q^+,\,p$ and $q^-$ are collinear.
But, this is possible only for finitely many points $p \in \partial\Omega_-$, thanks to the fact that the transport rays cannot intersect at an interior point.  
\end{proof}

In particular, if \,$\spt(f_-^+) \cap \spt(f_-^-) \neq \emptyset$, this requires a very special configuration of the boundary values - if $g \in C(\partial\Omega)$, then necessarily \,$p,\,q^\pm \in g^{-1}(t)$\, and the line segment $\overline{q^+ q^-}$ lies on a supporting line to $\partial\Omega_-$ at $p$; see the following example:

\begin{example}
Let $\Omega = B(0,2) \backslash \overline{B(0,1)}$. Take the boundary data equal to $g_-(x,y) = y$, for every $(x,y) \in \partial B(0,1)$, and 
$$g_+(x,y) = \threepartdef{-1}{\,\,y < -1,}{y}{\,\,y \in [-1,1],}{1}{\,\,y > 1.} $$\\
Then, it is easy to see that the solution to the BV least gradient problem exists and equals  
$$u (x,y) = \threepartdef{-1}{\,\,y < -1,}{y}{\,\,y \in [-1,1],}{1}{\,\,y > 1.} $$\\
Here, we see that \,$p \in \spt(f_-^+) \cap \spt(f_-^-) = \{ (0, \pm 1) \}$, $q^+=(-\sqrt{3},\pm1)$ and \,$q^-=(\sqrt{3},\pm1)$.
\end{example}
 
\begin{proposition}\label{stw:finitelymanychanges}
Let \,$\Omega \subset \mathbb{R}^2$ be an annulus. Suppose that $u \in BV(\Omega)$ is a least gradient function with trace $g \in BV(\partial\Omega)$. Then, $g_- \in BV(\partial\Omega_-)$ changes monotonicity finitely many times.
\end{proposition}

\begin{proof}
Set $f=\partial_\tau g$. There are two possibilities so that $g_-$ changes monotonicity: either at a point $p \in \spt(f_-^+) \cap \spt(f_-^-)$ or there is a flat part where $g_-$ is constant between $\spt(f_-^+)$ and $\spt(f_-^-)$. By Proposition \ref{stw:specialconfiguration}, the first variant can happen only finitely many times. We will argue by contradiction and assume that there are countably many flat parts $F_k^-$ of $g_-$. 

Fix any $\varepsilon > 0$. As there are countably many flat parts of $g_-$ and $\mathcal{H}^1(\partial\Omega_-)$ is finite, countably many of them have length smaller than $\varepsilon$. Now, take a flat part $F^-$ such that $\mathcal{H}^1(F^-) < \varepsilon$ and $\partial F^- = \{ p^+, p^- \}$, where $p^\pm \in \spt(f_-^\pm)$. Now, we make a similar argument as in the proof of Proposition \ref{stw:specialconfiguration}: consider the sets $V_n^\pm = \partial\Omega_- \cap B(p^\pm, \frac{1}{n})$. Then, inside any $V_n^\pm$, there is a point $p_n^\pm$ such that there is a transport ray coming out of $p_n^\pm$ to a point $q_n^\mp$ in $\partial\Omega_+$. 
Yet, we have
$$ |p_n^+ - q_n^-| + |p_n^- - q_n^+| \leq |p_n^+ - p_n^-| + |q_n^+ - q_n^-|.$$
Now, passing to the limit when $n \rightarrow \infty$, we have $p_n^\pm \rightarrow p^\pm$ and $q_n^\pm \rightarrow q^\pm$ where $q^\pm \in \partial\Omega_+$, and then  
$$ |p^+ - q^-| + |p^- - q^+| \leq |p^+ -p^-| + |q^+ -q^-| \leq \varepsilon + |q^+ - q^-|.$$
Yet, we have \,$\delta:=\mbox{dist}(\partial\Omega_+,\partial\Omega_-) >0$. Then, this means that there are two sequences $(q_k^\pm)_k \subset \partial\Omega_+$ such that the curves that connect $q_k^+$ to $q_k^-$ on $\partial\Omega_+$ are disjoint (thanks to the fact that the transport rays cannot intersect) and
$$\delta  \leq |q_k^+ -q_k^-|,$$
which is a contradiction since $\mathcal{H}^1(\partial\Omega_+) < +\infty$.
Finally, this means that there are only finitely many flat parts of $g_-$, so $g_-$ changes monotonicity only finitely many times. $\qedhere$
\end{proof}

The following Lemma, which follows from the strict convexity of $\partial\Omega_+$, will play a part in the proof to come (it is proved using a blow-up of $\partial\Omega_+$).

\begin{lemma}\label{lem:zawieranie}
(\cite[Lemma 3.8]{Gorny0}) Let $u \in BV(\Omega)$ be a least gradient function with trace $g$. Then, we have $\partial \{ u \geq t \} \cap \partial\Omega_+ \subset g_+^{-1}(t).$
\end{lemma}

The final issue concerns the images of the inner and outer boundary part under the boundary data $g$. This is important in view of the equivalence proved in Theorem \ref{thm:equivalence}; under the structural hypotheses (H1)-(H4) introduced in Section \ref{Sec. 4}, it enables us to find precisely the boundary data for which we have found a solution to the least gradient problem. 

\begin{lemma}\label{lem:images}
Let $\Omega \subset \mathbb{R}^2$ be an annulus. Suppose that $u \in BV(\Omega)$ is a least gradient function with trace $g \in C(\partial\Omega)$. Then, $g(\partial\Omega_-) \subset g(\partial\Omega_+)$.
\end{lemma}

\begin{proof}
As $\partial\Omega_\pm$ are compact and connected while $g$ is continuous, then the images $g(\partial\Omega_-)$ and $g(\partial\Omega_+)$ are intervals. Suppose that the inclusion does not hold; then choose $t \in g(\partial\Omega_-) \backslash g(\partial\Omega_+)$. Without loss of generality, assume that $t$ is greater than any element from $g(\partial\Omega_+)$. Consider the set $\partial \{ u \geq t \}$; if it is empty, then $\{u \geq t\} = \Omega$, which violates the trace condition on $\partial\Omega_+$. If it is not empty, Lemma \ref{lem:zawieranie} implies that the set $\partial E_t \cap \partial\Omega_+$ is empty; hence there is a line segment in $\partial E_t$ which has both ends in $\partial\Omega_-$, which is a contradiction. $\qedhere$
\end{proof}

\section{On the equivalence between the BV least gradient problem and the optimal transport}\label{Sec. 3}

The aim of this Section is to study the equivalences between the least gradient problem \eqref{BV least gradient problem on annulus}, the Beckmann problem \eqref{BeckmannIntro} and the classical Monge-Kantorovich problem \eqref{KantoIntro}. Throughout this Section, $\Omega \subset \mathbb{R}^2$ is assumed to be an annulus in the sense of Definition \ref{dfn:annulus}. Firstly, we show a relationship between solutions to the following problems:
\begin{equation} \label{Beckmann problem}
  \min \bigg\{ \int_{\bar{\Omega}} |v|\,:\, v \in \mathcal{M}(\overline{\Omega}; \mathbb{R}^2),\,\,\nabla\cdot v = f \bigg\}
\end{equation}
and 
\begin{equation} \label{shifted BV least gradient problem}
   \min\bigg\{ \int_\Omega |D u| \,:\,u \in BV(\Omega),\, \partial_\tau(Tu) = f \bigg\},
\end{equation}
where $\partial_\tau(Tu) = f$ is equivalent to saying that $Tu=g$ on $\partial\Omega$ for some $g$ such that $f = \partial_\tau g$, up to adding a constant on each connected component of $\partial\Omega$. The divergence condition in \eqref{Beckmann problem} is understood in the distributional sense: for every $\phi \in C^1(\bar{\Omega})$, we have $\int_{\bar{\Omega}} \nabla \phi \cdot \mathrm{d}v=\int_{\partial\Omega} \phi\,\mathrm{d}f$. In other words, we have $\nabla \cdot v = 0$ in $\Omega$\, and $v \cdot \nu_{|\partial\Omega} = f$. Moreover, the boundary condition in \eqref{shifted BV least gradient problem} is understood in the sense of traces. Furthermore, as $f$ is a tangential derivative of a $BV$ function on the closed sets $\partial\Omega_\pm$, it will be subject to a mass balance condition, i.e. 
$$ f_+(\partial\Omega_+) = f_-(\partial\Omega_-) = 0.$$
   
It is important to stress that while problem \eqref{Beckmann problem} is the usual Beckmann problem (also called the free material design problem), problem \eqref{shifted BV least gradient problem} is not the usual least gradient problem (i.e., the one with constraint $T u =g$). Here, we minimize $\int_\Omega |D u|$ over a wider range of boundary data. As $\partial\Omega$ is not connected, if we shift $g$ by a constant on any of the connected components of $\partial\Omega$, we change the boundary value in \eqref{shifted BV least gradient problem}, but it remains the same in \eqref{Beckmann problem}; hence, the formulation of \eqref{shifted BV least gradient problem} involves minimization over the set of all $g$ such that $f = \partial_\tau g$, i.e. $f$ is the tangential derivative of $g$. Clearly, if $u \in BV(\Omega)$ solves \eqref{shifted BV least gradient problem}, then it also solves the standard least gradient problem with boundary data $Tu$. We will come back to this issue at the end of Section \ref{Sec. 4}. \\



The main idea, coming from \cite{Gorny0}, is to take an admissible function $u$ in \eqref{shifted BV least gradient problem} and use its rotated gradient $v = R_{ \frac{\pi}{2}} \nabla u$; as in dimension two, a rotation of a gradient by $\frac{\pi}{2}$ is a divergence-free field in $\Omega$ and rotation interchanges the normal and tangent components at the boundary, this is an admissible vector field in \eqref{Beckmann problem}. This fact was shown as a step in the proof of \cite[Theorem 2.1]{Gorny0} and is formalized in the following proposition; we present the proof for completeness.

\begin{proposition}\label{prop:admissibility}
Suppose that $\Omega \subset \mathbb{R}^2$ is an open bounded set with Lipschitz boundary and let $u \in BV(\Omega)$ with trace $Tu = g$. Then, $v = R_{ \frac{\pi}{2}}\nabla u$ is a vector-valued measure such that $\nabla \cdot v = f$, where $f = \partial_\tau g$. In particular, it is an admissible function in \eqref{Beckmann problem}. 
\end{proposition} 
\begin{proof}
Let $u_n \in C^\infty(\Omega) \cap BV(\Omega)$ be a sequence which converges to $u$ in strict topology of $BV(\Omega)$, i.e. $u_n \rightarrow u$ in $L^1$ and $\int_\Omega |\nabla u_n| \, dx \rightarrow \int_\Omega |Du|$. 
We notice that the rotated gradients of $u_n$ have zero divergence inside $\Omega$, as for smooth functions
$$ \mbox{div}(R_{\frac{\pi}{2}} \nabla u_n) = \mbox{div}(-(u_n)_{x_2}, (u_n)_{x_1}) = -(u_n)_{x_2 x_1} + (u_n)_{x_1 x_2} = 0.$$
Integrating by parts, we get 
\begin{align*}
\int_\Omega R_{\frac{\pi}{2}} \nabla u_n(x) \cdot \nabla \phi(x)\,\mathrm{d}x &= \int_{\partial\Omega}R_{\frac{\pi}{2}} \nabla u_n (x)\cdot \nu(x)\,\phi(x)\,\mathrm{d}\mathcal{H}^1(x)\\
&= -\int_{\partial\Omega} Tu_n(x) \, \partial_\tau\phi(x)\,\mathrm{d}\mathcal{H}^1(x),\,\,\,\mbox{for all}\,\,\,\phi \in C^1(\bar{\Omega}).
\end{align*}
Yet, $\nabla u_n \rightharpoonup Du$ 
and the trace operator is continuous with respect to the strict convergence. Then, passing to the limit, we obtain $\nabla \cdot\left( R_\frac{\pi}{2} Du\right) = 0$\, in \,$\Omega$\, and \,$R_\frac{\pi}{2} Du \cdot \nu|_{\partial\Omega} = \partial_\tau (Tu) = f$. $\qedhere$
\end{proof} 
We point out that while Proposition \ref{prop:admissibility} does not require $f$ to be a measure, merely a continuous functional over $\text{Lip}(\partial\Omega)$, in this paper we require $f$ to be a measure supported on $\partial\Omega$ in order to obtain a converse result.\\

In the other direction, the authors of \cite{Gorny0} proved that if the domain $\Omega$ is strictly convex, a vector field $v \in L^1(\Omega, \mathbb{R}^2)$ admissible in \eqref{Beckmann problem} produces a function $u \in W^{1,1}(\Omega)$ admissible in \eqref{shifted BV least gradient problem}. However, their proof involves definition of $u$ as an integral of a certain 1-form; in our setting, it fails due to the fact that $\Omega$ is not simply-connected and the integral may depend on the choice of a path. In the next proposition, we use the result in the convex case to resolve this problem.

\begin{proposition}\label{prop:recovery}
Let $\Omega \subset \mathbb{R}^2$ be an annulus. Suppose that $v \in L^1(\Omega,\mathbb{R}^2)$ is such that $\nabla\cdot v = f$ in $\mathbb{R}^2$ as distributions, where $f \in \mathcal{M}(\partial\Omega)$ is a measure such that $f(\partial\Omega_\pm) = 0$. Then, there exists $u \in W^{1,1}(\Omega)$ such that $v = R_{\frac{\pi}{2}} \nabla u$. In particular,
$$ \int_\Omega |v|\,dx = \int_\Omega |\nabla u|\, dx.$$
Moreover, if \,$Tu = g$ then $f = \partial_\tau g$.
\end{proposition}

\begin{proof}
1. Denote $f_\pm = f|_{\partial\Omega_\pm}$. Let $v \in L^1(\Omega, \mathbb{R}^2)$ be such that $\nabla \cdot v = f$ in $\mathbb{R}^2$ as distributions. To be precise, if we take $\widetilde{v}$ to be a vector field in $L^1(\mathbb{R}^2,\mathbb{R}^2)$ defined as
$$ \widetilde{v} = \twopartdef{v}{\mbox{ in }\, \Omega,}{0}{ \mbox{ in }\, \mathbb{R}^2 \backslash {\Omega},}$$
then $\nabla \cdot \widetilde{v} = f$ as distributions in $\mathbb{R}^2$. We want to extend this vector field in a different way so that $\nabla \cdot v = f_+$. To this end, let $g_- \in BV(\partial \Omega_-)$ be such that $f_- = \partial_\tau g_-$ and take any $w \in W^{1,1}(\Omega_-)$ such that $T_{\partial\Omega_-} w = g_-$.

2. Now, we take the rotated gradient of $w$. Then, $v' = R_{ \frac{\pi}{2}} \nabla w \in L^1(\Omega_-, \mathbb{R}^2)$ is a vector field such that $\nabla \cdot v' = - f_-$ (the minus sign comes from the fact that the orientation of $\partial\Omega_-$ as a boundary of $\Omega_-$ is opposite to its orientation as a part of the boundary of $\Omega$). Let $\widetilde{v'}$ be an extension of $v'$ by $0$ to the whole of $\mathbb{R}^2$ as above, i.e.
$$ \widetilde{v'} = \twopartdef{v'}{\mbox{ in }\, \Omega_-,}{0}{ \mbox{ in }\, \mathbb{R}^2 \backslash {\Omega}_-.}$$
So, we have $\nabla \cdot (\widetilde{v} + \widetilde{v'}) = f_+$. Moreover, $(\widetilde{v} + \widetilde{v'})|_{\Omega} = v$.

3. Now, we use \cite[Proposition 2.1]{Gorny0} on $\Omega_+$ and obtain that there exists a function $\widetilde{u} \in W^{1,1}(\Omega_+)$ such that $\widetilde{v} + \widetilde{v'} = R_{\frac{\pi}{2}} \nabla \widetilde{u}$\, on $\Omega_+$. In particular, we have
$$ \int_{\Omega_+} |\widetilde{v} + \widetilde{v'}| \, dx = \int_{\Omega_+} |\nabla \widetilde{u}| \,  dx.$$
Moreover, $\partial_\tau (T_{\partial\Omega_+} \widetilde{u}) = f_+$. By applying again \cite[Proposition 2.1]{Gorny0} but this time on $\Omega_-$, we obtain also that
$\partial_\tau (T_{\partial\Omega_-} \widetilde{u}) = f_-$.
Now, set $u = \widetilde{u}|_{\Omega} \in W^{1,1}(\Omega)$. So, we see easily that 
$$ \int_{\Omega} |v| \, dx = \int_{\Omega} |\nabla u| \, dx.$$
Finally, the trace of $u$ is correct: clearly, $\partial_\tau(T_{\partial\Omega_+} u) = \partial_\tau (T_{\partial\Omega_+} \widetilde{u}) = f_+$. Moveover, as $|\nabla \widetilde{u}|(\partial\Omega_-) = 0$, the trace of $\widetilde{u}$ on $\partial\Omega_-$ from both sides coincides and, we have $T_{\partial\Omega_-} u = T_{\partial\Omega_-} \widetilde{u} = g_-$.
\end{proof}

When $v$ is merely a measure, we can employ a similar trick. However, we need one additional component: $v$ has to give no mass to the boundary. Otherwise, the trace of the obtained function $u$ would be incorrect (see also the discussion in \cite{Dweik}). So, we have the following:

\begin{proposition}\label{prop:measurerecovery}
Let $\Omega \subset \mathbb{R}^2$ be an annulus. Suppose that $v \in \mathcal{M}(\overline{\Omega}, \mathbb{R}^2)$ is such that $|v|(\partial\Omega) = 0$ and $\nabla\cdot v = f$ as distributions, where $f \in \mathcal{M}(\partial\Omega)$ is a measure such that $f(\partial\Omega_\pm) = 0$. Then, there exists $u \in BV(\Omega)$ such that $v= R_{\frac{\pi}{2}} D u$. In particular,
$$ \int_{\overline{\Omega}} |v| = \int_\Omega |D u|.$$
Moreover, if $Tu = g$ then $f = \partial_{\tau} g$. 
\end{proposition}

Now, we are ready to prove the equivalence of problems \eqref{Beckmann problem} and \eqref{shifted BV least gradient problem}. This boils down to two distinct problems: proving that the infima of these problems are equal and to constructing solutions of one problem from the other one.

\begin{theorem}\label{thm:equivalence}
We have \,$\inf \eqref{Beckmann problem} = \inf \eqref{shifted BV least gradient problem}$. Moreover, from each solution $u \in BV(\Omega)$ of \eqref{shifted BV least gradient problem}, one can construct a solution to \eqref{Beckmann problem}. In the other direction, from each solution $v \in \mathcal{M}(\overline{\Omega}, \mathbb{R}^2)$ of \eqref{Beckmann problem}, one can construct a solution to \eqref{shifted BV least gradient problem}, provided that $|v|(\partial\Omega) = 0$.
\end{theorem}
 
\begin{proof}
Suppose that $v_n \in L^1(\Omega, \mathbb{R}^2)$ is a minimizing sequence in \eqref{Beckmann problem}. By Proposition \ref{prop:recovery}, for each $n$, there exists $u_n \in W^{1,1}(\Omega)$ admissible in \eqref{shifted BV least gradient problem} such that $v_n = R_{\frac{\pi}{2}} \nabla u_n$. Hence
$$ \inf \eqref{Beckmann problem} \longleftarrow \int_\Omega |v_n| = \int_{\Omega} |\nabla u_n| \geq \inf \eqref{shifted BV least gradient problem}.$$
Conversely, suppose that $u_n \in BV(\Omega)$ is a minimizing sequence in \eqref{shifted BV least gradient problem}. By Proposition \ref{prop:admissibility}, the vector fields $v_n = R_{\frac{\pi}{2}} D u_n$ are admissible in \eqref{Beckmann problem}. Hence
$$ \inf \eqref{shifted BV least gradient problem} \longleftarrow \int_\Omega |D u_n| = \int_\Omega |v_n| \geq \inf \eqref{Beckmann problem}.$$
Hence, the two infima are equal. Now, we turn to the issue of constructing solutions of one problem from the other one.

Let $u \in BV(\Omega)$ be a minimizer of \eqref{shifted BV least gradient problem}. Let $v = R_{\frac{\pi}{2}}  D u$; by Proposition \ref{prop:admissibility}, it is an admissible vector field in \eqref{Beckmann problem}. Moreover, we have
$$ \inf \eqref{shifted BV least gradient problem} = \int_\Omega |D u| = \int_\Omega |v| \geq \inf \eqref{Beckmann problem}.$$
Hence, $v$ is a minimizer of \eqref{Beckmann problem}.
  
Finally, let $v \in \mathcal{M}(\overline{\Omega}, \mathbb{R}^2)$ be a minimizer of \eqref{Beckmann problem} such that $|v|(\partial\Omega) = 0$. By Proposition \ref{prop:measurerecovery}, there exists $u \in BV(\Omega)$ admissible in \eqref{shifted BV least gradient problem} and $v = R_{\frac{\pi}{2}} D u$. Yet, one has
$$ \inf \eqref{Beckmann problem} =   \int_\Omega |v| = \int_\Omega |D u| \geq \inf \eqref{shifted BV least gradient problem},$$
which implies that this function $u$ is, in fact, a minimizer for the problem \eqref{shifted BV least gradient problem}.   $\qedhere$
\end{proof}

In particular, a solution $v \in \mathcal{M}(\overline{\Omega}, \mathbb{R}^2)$ of the Beckmann problem which satisfies $|v|(\partial\Omega) = 0$ generates a function $u \in BV(\Omega)$ which solves the least gradient problem for boundary data \,$g = Tu$. If the solution to the Beckmann problem is unique, then also the boundary data $g$ for which we can construct the solution of the least gradient problem is unique up to adding the same constants on both connected components of $\partial\Omega$; we cannot solve the least gradient problem for a shifted value of $g$ if we added two different constants on $\partial\Omega_\pm$. 
\begin{example}
Let $\Omega = B(0,2) \backslash \overline{B(0,1)}$. Consider $f \equiv 0 \in \mathcal{M}(\partial\Omega)$ to be boundary data in the Beckmann problem. Functions $g \in BV(\partial\Omega)$ such that $f = \partial_\tau g$ are of the form
$$ g = \twopartdef{c_-}{\text{ on }\, \partial B(0,1),}{c_+}{\text{ on }\, \partial B(0,2).}$$
We consider such boundary data in the least gradient problem. The solution to the Beckmann problem is unique and equals $v \equiv 0$ in $\overline{\Omega}$. Then, Theorem \ref{thm:equivalence} gives us a constant solution to problem \eqref{shifted BV least gradient problem}. It is a solution to the least gradient problem with $c_+ = c_-$. However, for $c_+\neq c_-$, the least gradient problem admits no solution.
\end{example}

On the other hand, it is also possible to show equivalence between the Beckmann problem \eqref{BeckmannIntro} and the Monge-Kantorovich one \eqref{KantoIntro}, in the case where the domain $\Omega$ is an annulus. From \cite{San2,Villani},
the Kantorovich problem  
\begin{equation}\label{Kanto}
    \min\bigg\{\int_{\bar{\Omega} \times \bar{\Omega}} |x-y|\,\mathrm{d}\gamma\,:\,\gamma \in  \mathcal{M}^+(\bar{\Omega} \times \bar{\Omega}),\,(\Pi_x)_{\#}\gamma=f^+ \,\,\mbox{and}\,\,(\Pi_y)_{\#} \gamma =f^-\bigg\}
\end{equation}
admits a dual formulation:
\begin{equation}\label{Kanto dual}
    \sup\bigg\{\int_{\bar{\Omega}} \phi\,\mathrm{d}(f^+ - f^-)\,:\,\phi \in  \Lip_1(\Omega)\bigg\}. 
\end{equation}
In fact, we have
$$  \min\bigg\{\int_{\bar{\Omega} \times \bar{\Omega}} |x-y|\,\mathrm{d}\gamma\,:\,\gamma \in  \mathcal{M}^+(\bar{\Omega} \times \bar{\Omega}),\,(\Pi_x)_{\#}\gamma=f^+ \,\,\mbox{and}\,\,(\Pi_y)_{\#} \gamma =f^-\bigg\}$$
$$=\min_{\gamma \in  \mathcal{M}^+(\bar{\Omega} \times \bar{\Omega})}\bigg\{\int_{\bar{\Omega} \times \bar{\Omega}} |x-y|\,\mathrm{d}\gamma \,+ \sup_{\phi^\pm \in C(\bar{\Omega})}\bigg\{\int_{\bar{\Omega}} \phi^+ \,\mathrm{d}f^+ - \int_{\bar{\Omega}} \phi^-\,\mathrm{d}f^-  -\int_{\bar{\Omega} \times \bar{\Omega}}\left[\phi^+(x) - \phi^-(y)\right]\mathrm{d}\gamma\bigg\} \bigg\}$$ 
$$=\min_{\gamma \in  \mathcal{M}^+(\bar{\Omega} \times \bar{\Omega})}\bigg\{\sup_{\phi^\pm \in C(\bar{\Omega})}\bigg\{\int_{\bar{\Omega} \times \bar{\Omega}} \left[|x-y|-(\phi^+(x) - \phi^-(y))\right]\,\mathrm{d}\gamma(x,y) + \int_{\bar{\Omega}} \phi^+ \,\mathrm{d}f^+ - \int_{\bar{\Omega}} \phi^-\,\mathrm{d}f^-  \bigg\} \bigg\}.$$\\
By a formal inf-sup exchange, we get
$$=\sup_{\phi^\pm \in C(\bar{\Omega})}\bigg\{\min_{\gamma \in  \mathcal{M}^+(\bar{\Omega} \times \bar{\Omega})}\bigg\{\int_{\bar{\Omega} \times \bar{\Omega}} \left[|x-y|-(\phi^+(x) - \phi^-(y))\right]\,\mathrm{d}\gamma(x,y) \bigg\} + \int_{\bar{\Omega}} \phi^+ \,\mathrm{d}f^+ - \int_{\bar{\Omega}} \phi^-\,\mathrm{d}f^-\bigg\}.$$
Yet,
$$ \min_{\gamma \in  \mathcal{M}^+(\bar{\Omega} \times \bar{\Omega})}\bigg\{\int_{\bar{\Omega} \times \bar{\Omega}} \left[|x-y|-(\phi^+(x) - \phi^-(y))\right]\,\mathrm{d}\gamma(x,y) \bigg\} =\begin{cases}
0  & \,\,\mbox{if}\,\,\,\,\phi^+(x) - \phi^-(y) \leq |x-y|,\\
- \infty & \,\,\mbox{else}.
\end{cases}$$
Finally, this yields that
$$\min\eqref{Kanto}$$
$$=\sup\bigg\{ \int_{\bar{\Omega}} \phi^+ \,\mathrm{d}f^+ - \int_{\bar{\Omega}} \phi^-\,\mathrm{d}f^-\,:\,\phi^\pm \in C(\bar{\Omega}),\,\phi^+(x) - \phi^-(y) \leq |x-y|\bigg\}.$$
But now, it is clear that we can assume $\phi^+(x):=\min\{|x-y| + \phi^-(y)\,:\,y \in \bar{\Omega}\}$, for every $x \in \bar{\Omega}$, and so, $\phi^-=\phi^+$.
From this duality result $\min \eqref{Kanto}=\sup \eqref{Kanto dual}$, we infer that optimal $\gamma$ and $\phi$ satisfy the following equality:
$$\int_{\bar{\Omega} \times \bar{\Omega}}\left[|x-y| - (\phi(x) - \phi(y))\right]\,\mathrm{d}\gamma(x,y)=0,$$
which implies that 
$$\phi(x) - \phi(y) = |x-y|\,\,\,\,\mbox{on}\,\,\,\,\spt(\gamma).$$
Let us introduce the following: 
\begin{definition}
We call transport ray any maximal segment \,$\left[x,y\right]$ satisfying \,$\phi(x)-\phi(y)=|x-y|$.
\end{definition}
 Following this definition, we see that an optimal
transport plan $\gamma$ has to move the mass along the transport rays.\\

Now, we prove equivalence between the Kantorovich problem \eqref{Kanto} and the Beckmann one \eqref{Beckmann problem}.
\begin{proposition} \label{Prop. 3.3}
Suppose that all the transport rays between $f^+$ and $f^-$ are inside the annulus \,$\Omega$. Let $\gamma$ be an optimal transport plan for \eqref{Kanto} and let us define the vectorial measure $w_\gamma$ as follows:
$$<w_\gamma,\xi>:=\int_{\bar{\Omega} \times \bar{\Omega}} \int_0^1 \xi((1-t)x+ty) \cdot (y - x) \,\,\mathrm{d}t\,\mathrm{d}\gamma (x,y),\,\,\,\,\mbox{for all}\,\,\,\,\xi \in C(\bar{\Omega},\mathbb{R}^2).$$\\
Then, $w_\gamma$ solves \eqref{Beckmann problem}. Moreover, we have $\min\eqref{Beckmann problem}=\sup\eqref{Kanto dual}=\min\eqref{Kanto}$.
\end{proposition}
\begin{proof}
First, we see easily that $w_\gamma$ is admissible in \eqref{Beckmann problem} (this follows immediately by taking as a test function $\xi=\nabla \phi$). On the other hand, we have 
$$|w_\gamma|(\bar{\Omega}) \leq \int_{\bar{\Omega} \times \bar{\Omega}} |x-y|\,\mathrm{d}\gamma=\min \eqref{Kanto}=\sup \eqref{Kanto dual}.$$
Let $v$ be an admissible flow in \eqref{Beckmann problem} and let $\phi$ be a $C^1$ function such that $|\nabla \phi| \leq 1$. Then, one has 
$$\int_{\bar{\Omega}} \phi \,\mathrm{d}(f^+ - f^-)=\int_{\bar{\Omega}} \nabla \phi \cdot \mathrm{d}v \leq \int_{\bar{\Omega}} |v|.$$
This implies that 
$$ \sup \eqref{Kanto dual} \leq \min \eqref{Beckmann problem}.$$\\
Consequently, we get that $w_\gamma$ is a solution for \eqref{Beckmann problem}. And, we have $\min\eqref{Beckmann problem}=\sup\eqref{Kanto dual}=\min\eqref{Kanto}$.  
\end{proof} 
In addition, following \cite[Chapter 4]{San2} and using Proposition \ref{Prop. 3.3}, we are able to prove that every solution $w$ for the Beckmann problem \eqref{Beckmann problem} is of the form $w=w_\gamma$, for some optimal transport plan $\gamma$ for \eqref{Kanto}.
 
On the other hand, one can associate with $w_\gamma$ a scalar positive measure $\sigma_\gamma$ (which is called {\it{transport density}}):
$$<\sigma_\gamma,\varphi>:=\int_{\bar{\Omega} \times \bar{\Omega}} \int_0^1 \varphi((1-t)x+ty) |x-y| \,\,\mathrm{d}t\,\mathrm{d}\gamma (x,y),\,\,\,\,\mbox{for all}\,\,\,\,\varphi \in C(\bar{\Omega}).$$
Moreover, it is not difficult to see that if $\phi$ is a Kantorovich potential, between $f^+$ and $f^-$, then we have the following:
$$w_\gamma=-\sigma_\gamma \nabla \phi.$$
In this way, we get existence of a solution for the least gradient problem \eqref{shifted BV least gradient problem} as soon as the transport density $\sigma_\gamma$ gives zero mass to the boundary $\partial\Omega$ (i.e., $\sigma_\gamma(\partial\Omega)=0$). Moreover, we get uniqueness of the solution $u$ for \eqref{shifted BV least gradient problem} if $\sigma_\gamma$ does not depend on the choice of $\gamma$. We note that the uniqueness of the transport density $\sigma_\gamma:=\sigma$ holds as soon as $f^+$ or $f^-$ is in $L^1(\Omega)$, which is not the case here since $f$ is singular (it is supported on the boundary). Yet, we will show uniqueness of $\sigma$ under some  assumptions on $\Omega,\,f^+$ and $f^-$.
\section{Least gradient problem: existence and uniqueness}\label{Sec. 4}
In this section, we will prove that on an annulus $\Omega \subset \mathbb{R}^2$, under some admissibility assumptions on the boundary datum $g$, the least gradient problem
\begin{equation} \label{shifted BV least gradient problem original}
   \min\bigg\{ \int_{\Omega} |D u| \,:\,u \in BV(\Omega),\,T u=g \bigg\} 
\end{equation}
has a solution. We recall that we need to restrict to dimension $2$ because only in this framework we can use rotated gradients, and they have zero divergence. In addition, we will assume that $g \in BV(\partial\Omega)$ since, in this way, one has the equivalence between the Beckmann problem \eqref{Beckmann problem} and a version of the least gradient problem \eqref{shifted BV least gradient problem}. We start with proving existence of a solution to the Beckmann problem which gives no mass to the boundary and then pass through problem \eqref{shifted BV least gradient problem} to the least gradient problem \eqref{shifted BV least gradient problem original}. 

First of all, let us introduce our admissibility conditions. These are formally conditions on a Dirichlet datum $g$ in the least gradient problem; however, as they do not depend on the exact values of $g$, only on its structure and total variation, we may think of them equivalently as conditions on its tangential derivative $f = \partial_\tau g$:\\ 
\\
(H1)\,\,\,\,$g \in BV(\partial\Omega ).$\\ \\
(H2)\,\,\,\,$\partial\Omega_\pm$ can be decomposed into parts $(\chi_i^\pm)_i,\,(\Gamma_i^\pm)_i$ and $(F_i^\pm)_i$ such that\,:\\ \\ 
$\bullet\,\,\,$ On each $\chi_i^\pm$ (resp. $\Gamma_i^\pm$) the boundary datum $g$ is increasing (resp. decreasing) with $TV(g \res \chi_i^+) = TV(g\res \chi_i^-)$ and $TV(g \res \Gamma_i^+) = TV(g\res \Gamma_i^-)$.\\ \\
$\bullet\,\,\,$ For every $i$, we have $[\partial_{\tau}g_+](F_i^+)=0$ (this means that $F_i^+$ is a flat part or $F_i^+:=F_i^{++} \cup F_i^{+-}$ where $g$ is increasing on $F_i^{++}$ and decreasing on $F_i^{+-}$ with $TV(g_{|F_i^{++}})=TV(g_{|F_i^{+-}})$)
and \,$g_-$ is constant on $F_i^-$.\\ \\
$\bullet\,\,\,$ 
Between each two curves $\chi^\pm_i$ and $\Gamma^\pm_i$, there is a (flat) part $F_i^\pm$.\\ \\
In addition, we want to add a condition on $g$ which will be necessary to guarantee that all the transport rays are inside $\bar{\Omega}$. Before that, we need to introduce the following: 
\begin{definition} 
Let $\Gamma^\pm$ be two arcs on $\partial\Omega_\pm$. Then, we say that $\Gamma^+$ is visible from $\Gamma^-$ if the following holds:
$$\mbox{for all}\,\,\,x \in \Gamma^+\,\,\mbox{and}\,\,\,y \in \Gamma^-,\,\,\,\mbox{we have}\,\,[x,y] \subset \bar{\Omega}.$$
\end{definition}
\noindent So, our visibility condition should be the following:\\ \\ (H3)\,\,\,\,For every $i$, $\chi_i^+$ is visible from $\chi_i^-$, $\Gamma_i^+$ is visible from $\Gamma_i^-$ and $F_i^{++}$ is visible from $F_i^{+-}$.\\ \\
The second condition that we need so that all the transport rays lie inside $\bar{\Omega}$ is 
an inequality linking the locations of $\chi_i^\pm$,  $\Gamma_i^\pm$ and $F_i^{+\pm}$. Set $\Lambda^+=\bigcup_i \chi_i^+ \cup \Gamma_i^- \cup F_i^{++}$ and $\Lambda^-=\bigcup_i \chi_i^- \cup \Gamma_i^+ \cup F_i^{+-}$, then we assume\\ \\ 
$\mbox{(H4)\,\,\,For every}\,\,\,i,\,\,\mbox{we have the following:}$\\ \\
$\bullet\,\,\,d_M(\chi_i^+,\chi_i^-) +  d_M(\Lambda^+  \backslash \chi_i^+, \Lambda^-  \backslash \chi_i^-) <  \text{dist}(\chi_i^+,\Lambda^- \backslash \chi_i^-) +  \text{dist}(\chi_i^-,\Lambda^+  \backslash \chi_i^+),$ \\
$\bullet\,\,\,d_M(\Gamma_i^+ , \Gamma_i^-) +  d_M(\Lambda^+ \backslash \Gamma_i^-,\Lambda^- \backslash\Gamma_i^+)  <  \text{dist}(\Gamma_i^+,\Lambda^+ \backslash \Gamma_i^-) +  \text{dist}(\Gamma_i^-,\Lambda^- \backslash \Gamma_i^+),$\\
$\bullet\,\,\,d_M(F_i^{++},F_i^{+-}) +  d_M(\Lambda^+  \backslash F_i^{++}, \Lambda^-  \backslash F_i^{+-}) <  \text{dist}(F_i^{++},\Lambda^- \backslash F_i^{+-}) +  \text{dist}(F_i^{+-},\Lambda^+  \backslash F_i^{++}),$ \\ \\
where $d_M(\Gamma,\Gamma^\prime)$ denotes the maximal distance between two arcs $\Gamma$ and $\Gamma^\prime$, and $\mbox{dist}(\Gamma,\Gamma^\prime)$ is the minimal distance between them.\\ \\
Under the assumptions (H1), (H2), (H3) $\&$ (H4), we have the following:\\
\begin{lemma}\label{transport rays inside the annulus}
Set $f=\partial_\tau g$. Then, all the transport rays between $f^+$ and $f^-$ lie inside the annulus $\Omega$. More precisely, any transport ray \,$R$ is of the form $[x,y]$ with $x \in \chi_i^+$ and $y \in \chi_i^-$, $x \in \Gamma_i^-$ and $y \in \Gamma_i^+$ or \,$x,\,y \in F_i^+$, for some $i$.
\end{lemma}
\begin{proof}
Let $R:=[x,y]$ be a transport ray. As $x \in \spt(f^+)$,
then $x \in \chi_i^+,\,\Gamma_i^-$ or $F_i^{++}$, for some $i$.
Suppose that $x \in \chi_i^+$ and $y \notin \chi_i^-$. As $TV(g_{|\chi_i^+})=TV(g_{|\chi_i^-})$, then there exists a transport ray $R^\prime:=[x^\prime,y^\prime]$ with $y^\prime \in \chi_i^-$ and $x^\prime \in \Lambda^+  \backslash \chi_i^+$. In particular, we have $(x,y),\,(x^\prime,y^\prime) \in \spt(\gamma)$, where $\gamma$ is an optimal transport plan for \eqref{Kanto}. Let $\phi$ be a Kantorovich potential between $f^+$ and $f^-$. Then, we have (this is the so-called {\it{cyclical monotonicity property}}):
$$|x-y| + |x^\prime - y^\prime|=\phi(x) - \phi(y) + \phi(x^\prime) - \phi(y^\prime) \leq |x-y^\prime| + |x^\prime - y|. $$
Yet,
$$|x-y^\prime| + |x^\prime - y| \leq d_M(\chi_i^+,\chi_i^-) +  d_M(\Lambda^+  \backslash \chi_i^+, \Lambda^- \backslash \chi_i^-)$$
and 
$$ \text{dist}(\chi_i^+,\Lambda^- \backslash \chi_i^-) +  \text{dist}(\chi_i^-,\Lambda^+  \backslash \chi_i^+) \leq |x-y| + |x^\prime - y^\prime|.$$\\
This contradicts the assumption (H4). The other cases can be treated in a similar way. 
\end{proof}

We also want to study the uniqueness of the solution of \eqref{Beckmann problem}. For this aim, we will prove the uniqueness of the optimal transport plan $\gamma$ in \eqref{Kanto}. More precisely, we have (the proof is essentially based on some arguments used in \cite[Proposition 2.5]{Dweik}):   
\begin{proposition}  \label{Prop. 4.2} 
Under the assumption that \,$\Omega_\pm$ are strictly convex, 
there is a unique optimal transport plan
$\gamma$ for \eqref{Kanto}, between $f^+$ and $f^-$, which will be induced by a transport map $S$, provided that $f^+$ is atomless.
\end{proposition}
\begin{proof}
Let $\gamma$ be an optimal transport plan between $f^+$ and $f^-$. Let $D$ be the set of points whose belong to several transport rays. Fix $x \in \Lambda^+ \cap D$ and let $R^\pm_x$ be
two different transport rays starting from
$x$. By Lemma \ref{transport rays inside the annulus}, we have, under the assumptions (H1), (H2), (H3) \& (H4), that all the transport rays between $f^+$ and $f^-$ lie inside $\bar{\Omega}$. 
There are three possibilities for $x$: $x\in \chi_i^+,\,\Gamma_i^-$ or $F_i^{++}$. Moreover, thanks to Lemma \ref{transport rays inside the annulus}, if $x \in \chi_i^+$, then both transport rays $R_x^+$ and $R_x^-$ should end on $\chi_i^-$, while if $x \in \Gamma_i^-$, then both transport rays $R_x^+$ and $R_x^-$ should end on $\Gamma_i^+$ and if $x \in F_i^{++}$, then both transport rays $R_x^+$ and $R_x^-$ should end on $F_i^{+-}$.  Let $∆_x \subset \Omega$ be the region delimited by
$R_x^+$, $R_x^-$
and $\partial\Omega$.
Then, we see easily that the sets $\{\Delta_x\}_{x \in D}$ must be disjoint with $|\Delta_x| >0$, for every $x \in D$. This implies that the set $D$ is at most countable. Yet, $f^+$ is atomless and so, $f^+(D)=0$. In addition, taking into account that $\Omega_\pm$ are strictly convex, we have that, for $f^+-$almost every
$x \,\notin D$,
there is a unique transport ray $R_x$ starting from
$x$, and
this ray $R_x$
intersects $\spt(f^-)$
at exactly one point $S(x)$. This implies that $\gamma=(Id,S)_{\#}f^+$. Yet, this is sufficient to infer that $\gamma$ is the unique optimal transport plan for \eqref{Kanto} since, if $\gamma^\prime$ is another optimal transport plan then \,$\gamma^{\prime\prime}=(\gamma + \gamma^\prime)/2$\, is also optimal for \eqref{Kanto}, which is not possible as $\gamma^{\prime\prime}$ must be induced by a transport map.   $\qedhere$
\end{proof}

Now, we are ready to state our main result concerning the Beckmann problem.
\begin{theorem}\label{thm:existence}
Assume that $\Omega \subset \mathbb{R}^2$ is an annulus. Let $f = \partial_\tau g$, where $g$ satisfies the admissibility conditions (H1)-(H4).
Then the Beckmann problem \eqref{Beckmann problem} admits a solution $v \in \mathcal{M}(\bar{\Omega}, \mathbb{R}^2)$ and $|v|(\partial\Omega) = 0$. Moreover, if $f^+$ is atomless, then the solution is unique.  
\end{theorem}
\begin{proof}
Let $f$ be the tangential derivative of the boundary datum $g$, i.e. $f=\partial_\tau g$, where $g$ satisfies the admissibility conditions (H1)-(H4). Let $\gamma$ be an optimal transport plan for \eqref{Kanto}. By Lemma \ref{transport rays inside the annulus} and Proposition \ref{Prop. 3.3}, one can construct a minimizer $v_\gamma$ for \eqref{Beckmann problem}. Now, we only need to show that $|v_\gamma|(\partial\Omega)=0$.
Yet, recalling the construction of $v_\gamma$, we have
$$
|v_\gamma|(\partial\Omega)=\int_{\bar{\Omega} \times \bar{\Omega}} \mathcal{H}^1(\partial\Omega \cap \left[x,y\right])\,\mathrm{d}\gamma(x,y).$$
As \,$\Omega_\pm$ are strictly convex, we infer that $|v_\gamma|(\partial\Omega)=0$. For uniqueness, it is enough to see that by Proposition \ref{Prop. 4.2}, there is a unique optimal transport plan $\gamma$ for \eqref{Kanto}. But, we recall that every solution $v$ for \eqref{Beckmann problem} is of the form $v=v_\gamma$, for some optimal transport plan $\gamma$. This implies that $v$ is the unique solution for \eqref{Beckmann problem}.  $\qedhere$
\end{proof}

Now, we want to go back to the least gradient problem. The first step is to construct a solution to the auxiliary problem \eqref{shifted BV least gradient problem} and translate it to a solution of the usual least gradient problem \eqref{BV least gradient problem classical} for some fixed boundary data $g$.

\begin{theorem}
Suppose that $\Omega \subset \mathbb{R}^2$ is an annulus and that $f = \partial_\tau g$, where $g$ satisfies the admissibility conditions (H1)-(H4). Then there exists a solution to problem \eqref{shifted BV least gradient problem}. Moreover, there exists $\widetilde{g} \in BV(\partial\Omega)$ such that $f = \partial_\tau \widetilde{g}$ such that there exists a solution to the least gradient problem \eqref{shifted BV least gradient problem original} with boundary data $\widetilde{g}$. If $g \in C(\partial\Omega)$, then the solutions to both problems are unique.
\end{theorem}

\begin{proof}
By Theorem \ref{thm:existence}, there exists a solution $v \in \mathcal{M}(\bar{\Omega}, \mathbb{R}^2)$ to the Beckmann problem \eqref{Beckmann problem} with $|v|(\partial\Omega) = 0$ (in addition, this solution is unique as soon as $f^+$ is atomless). Then, by Theorem \ref{thm:equivalence}, there exists a function $u \in BV(\Omega)$ which is a solution of the auxiliary problem \eqref{shifted BV least gradient problem} (which is also unique if $g$ is continuous, as then $f$ is atomless). Let $\widetilde{g} = Tu$. As the infimum in \eqref{shifted BV least gradient problem} is taken with respect to all possible traces with tangential derivative $f$, i.e. functions of the form $\widetilde{g} + \lambda_- \chi_{\partial\Omega_-} + \lambda_+ \chi_{\partial\Omega_+}$, so in particular $u$ is a solution to the least gradient problem \eqref{shifted BV least gradient problem original} with boundary data $\widetilde{g}$.  $\qedhere$
\end{proof}

In other words, what happens in the above Theorem is that when we use Theorem \ref{thm:equivalence}, we have no control on the vertical shifts of the boundary data by a constant on each connected component of $\partial\Omega$. Therefore, we are able to prove existence of a solution to the least gradient problem with some boundary data $\widetilde{g}$ (which differs from the original function $g$ by a constant on each connected component of $\partial\Omega$), but without calculating directly the minimizer $u$ in the auxiliary problem \eqref{shifted BV least gradient problem} it may be hard to compute $\widetilde{g}$. However, under an additional constraint on the total variation, the following proposition enables us to identify the boundary data $\widetilde{g}$ given by the previous theorem without having to first calculate the solution $u$ of problem \eqref{shifted BV least gradient problem}.
\begin{proposition}\label{prop:backtoleastgradient}
Suppose that $f$ satisfies assumptions (H1)-(H4) and $|f_-|(\partial\Omega_-) = |f_+|(\partial\Omega^+)$. Then, there is a unique $g \in BV(\partial\Omega)$ such that $f = \partial_\tau g$ and there exists a solution to the least gradient problem with this boundary data $g$. Moreover, if \,$f$ is atomless then the solution is unique.
\end{proposition}

\begin{proof}
Fix $x_\pm = \chi_1^\pm \cap F_1^\pm$. For $t_\pm \in \partial\Omega_\pm$, we set $g_\pm(t_\pm) = \int_{x_\pm}^{t_\pm} f$ (the integral is taken so that the tangent vector moves counterclockwise). Then, $f = \partial_\tau g$ and, by assumption (H2) and the equality of masses, $g_\pm$ change by the same value on each $\chi_i^\pm$, $\Gamma_i^\pm$ and is constant on each $F_i^\pm$. Then $g$ is the only function with tangential derivative $f$ such that $g_-(\partial\Omega_-) \subset g_+(\partial\Omega_+)$. Now, Theorems \ref{thm:equivalence} and \ref{thm:existence} give us existence of a least gradient function $u \in BV(\Omega)$ which solves \eqref{shifted BV least gradient problem}. However, due to Lemma \ref{lem:zawieranie} traces of least gradient functions satisfy $g_-(\partial\Omega_-) \subset g_+(\partial\Omega_+)$; hence $Tu = g$. Uniqueness is guaranteed by Theorem \ref{thm:existence}.  $\qedhere$
\end{proof}

Finally, we illustrate the results in this Section with the following Example:

\begin{example}
Let $\Omega = B(0,2) \backslash \overline{B(0,1)}$. Let $g_- \in C(\partial\Omega_-) \cap BV(\partial\Omega_-)$ be defined as follows:
$$ g_-(x,y) = \begin{cases} 
1 & \,\,\mbox{if}\,\,\,\, \,y > \frac{1}{2},\\
0 & \,\,\mbox{if}\,\,\,\, y < - \frac{1}{2},\\
y +\frac{1}{2}  &\,\, \mbox{else}. \end{cases}$$
Similarly, let $g_+ \in C(\partial\Omega_+) \cap BV(\partial\Omega_+)$ be defined as follows:
$$ g_+(x,y) = \begin{cases} 
1 & \,\,\mbox{if}\,\,\,\, \,y > 1,\\
0 & \,\,\mbox{if}\,\,\,\, y < - 1,\\
\frac{1}{2}(y+1) &\,\, \mbox{else}. \end{cases}$$
We check the admissibility conditions (H1)-(H4). By definition, $g \in BV(\partial\Omega)$. Moreover, we can decompose $\partial\Omega$ as in (H2): in the notation introduced at the beginning of Section 4, let us call the arc where $g_\pm$ is increasing  $\chi^\pm$ and the arc where $g_\pm$ is  decreasing $\Gamma^\pm$ (we drop the index $i$ as there is only one such arc). We call the remaining arcs, on which $g$ is constant, $F_1^\pm$ (with $g \equiv 1$ on $F_1^\pm$) and $F_2^\pm$ (with $g \equiv 0$ on $F_2^\pm$). In addition, $TV(g \res \chi^-) = 1 = TV(g \res \chi^+)$ and $TV(g \res \Gamma^-) = 1 = TV(g \res \Gamma^+).$ The situation is presented on Figure \ref{fig:prototype}.

As for the visibility condition (H3), we check that the tangent line to the inner circle $\partial B(0,1)$ at $(\frac{\sqrt{3}}{2}, \pm \frac{1}{2})$ crosses the outer circle $\partial B(0,2)$ at $(\sqrt{3}, \mp 1)$, hence $\chi^-$ is visible from $\chi^+$; similarly, $\Gamma^-$ is visible from $\Gamma^+$. 

Finally, we look at condition (H4). The idea behind it is such that the transport should take place between $\chi^\pm$ and between $\Gamma^\pm$, so that transport rays lie inside $\Omega$. Now, fix four points which are ends of two transport rays (of which we may think as points in the preimage $g^{-1}(t)$) $p_\pm \in \chi^\pm$ and $q_\pm \in \Gamma^\pm$. The visibility conditions enforce that the transport rays between these points are $p_- p_+$ and $q_- q_+$; we have to make sure that it is in fact the shortest connection possible between these four points. In other words, we have to check that
$$ |p_- - p_+| + |q_- - q_+| < |p_- - q_-| + |p_+ - q_+|,$$
as we can exclude the connection between points in $\chi^-$ and \,$\Gamma^+$, because both sets lie in the support of $f^-$. First, we see that
$$ |p_- - p_+| + |q_- - q_+| \leq d_M(\chi^+, \chi^-) + d_M(\Gamma^+, \Gamma^-)$$
and
$$ \text{dist}(\chi^-, \Gamma^-) + \text{dist}(\chi^+, \Gamma^+) \leq |p_- - q_-| + |p_+ - q_+|.$$
\\
Yet, we have $d_M(\chi^+, \chi^-) = d_M(\Gamma^+, \Gamma^-) = \text{ dist}(\chi^-, \Gamma^-) = \sqrt{3}$ and $\text{dist}(\chi^+, \Gamma^+) = 2 \sqrt{3}$. Hence, 
$$ d_M(\chi^+, \chi^-) + d_M(\Gamma^+, \Gamma^-) = 2\sqrt{3} < 3\sqrt{3} = \text{dist}(\chi^-, \Gamma^-) + \text{dist}(\chi^+, \Gamma^+) $$
and (H4) holds. Hence, by Theorem \ref{thm:existence} there exists a solution to the Beckmann problem with boundary data $f = \partial_\tau g$ and, by Proposition \ref{prop:backtoleastgradient}, there exists a unique solution to the least gradient problem with boundary data $g$.

\begin{figure}[h]
    \includegraphics[scale=0.17]{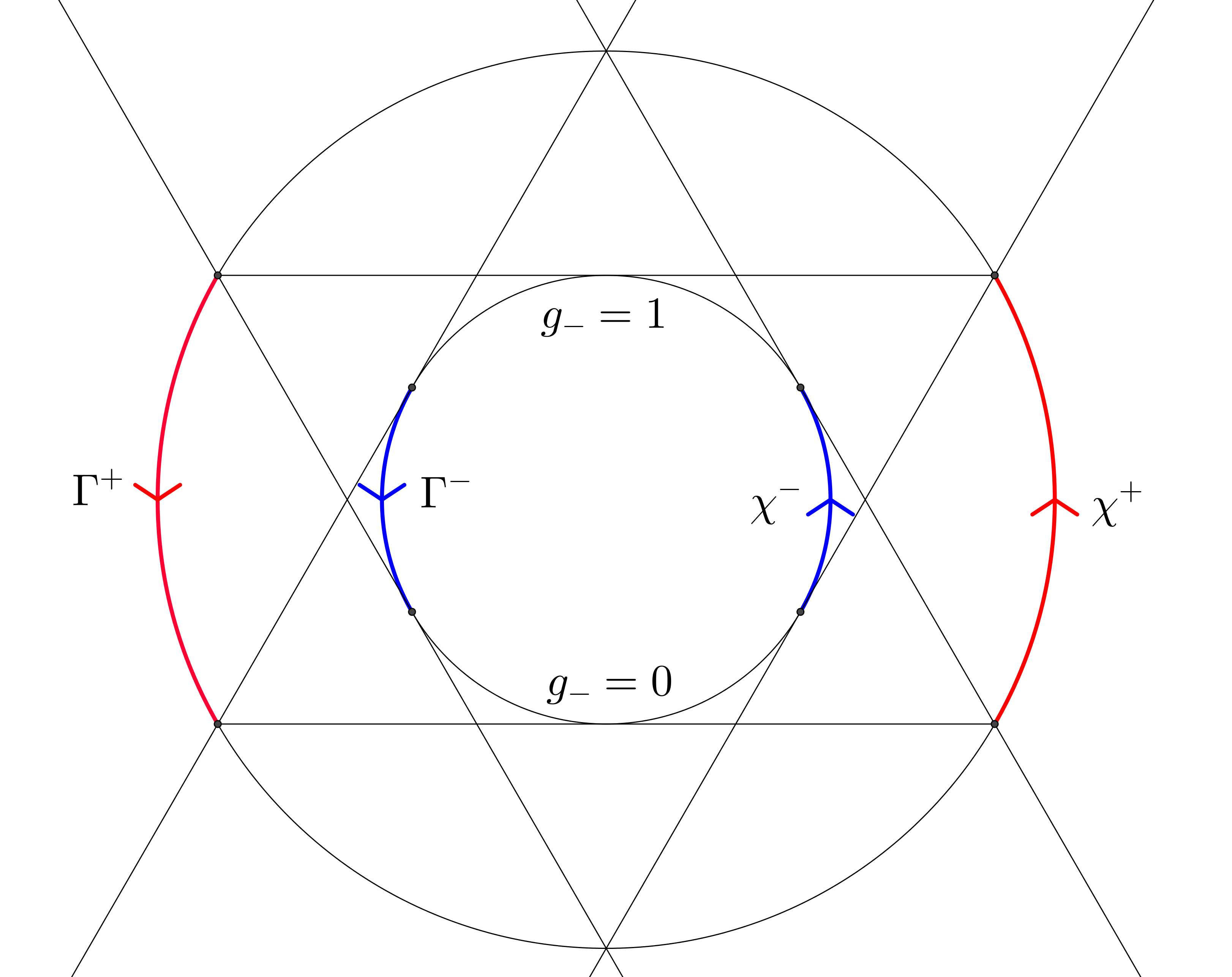}
    \caption{Visibility conditions in practice}
    \label{fig:prototype}
\end{figure}

\end{example}

\section{$W^{1,p}$ regularity of the solution to the least gradient problem}\label{Sec. 5}
The aim of this section is to study the $W^{1,p}$ regularity of the solution $u$ of the least gradient problem \eqref{shifted BV least gradient problem} in the case where the domain $\Omega$ is an annulus. First, we note that this question has already considered in \cite{Dweik}, but in what concerns the case where the domain $\Omega$ is uniformly convex, where the authors proved the following statement:
$$ g \in W^{1,p}(\partial\Omega) \Rightarrow u \in W^{1,p} (\Omega),\,\,\mbox{for every}\,\,\,p \leq 2.$$ 
In addition, they introduce a counter-example to the $W^{1,p}$ regularity of \,$u$\, for $p>2$. More precisely, it is possible to construct a Lipschitz function $g$ on the boundary so that the corresponding solution $u$ of the BV least gradient problem is not in $W^{1,2 +\varepsilon}(\Omega)$, for every $\varepsilon >0$. Recalling the relationship between the solution $u$ of the BV least gradient problem \eqref{shifted BV least gradient problem} and the minimizer $v$ of the Beckmann problem \eqref{Beckmann problem}, that is $v=R_{\frac{\pi}{2}} D u$, we see that studying the $W^{1,p}$ regularity of $u$ is equivalent to study the $L^p$ summability of the transport density \,$\sigma=|v|$. The difficulty, here, is that the measures $f^+$ and $f^-$ are concentrated on the boundary $\partial\Omega$ (and so, they are singular). As a consequence of that, we cannot use, for instance, the results of \cite{DePas1,DePas2,DePas3,San09} about the $L^p$ summability of the transport density between two $L^p$ densities $f^+$ and $f^-$ on $\Omega$. Moreover, the authors of \cite{DweSan} have considered the case where the source measure $f^+ \in L^p(\Omega)$ while the target one $f^-$ is the projection of $f^+$ into the boundary. In this case, they show that the transport density $\sigma$ is in $L^p(\Omega)$ provided $f^+ \in L^p(\Omega)$ and $\Omega$ satisfies an exterior ball condition; unfortunately, this is not the situation here. 
However, under the assumption that the domain $\Omega$ is uniformly convex, the authors of \cite{Dweik} show that the transport density $\sigma$ should be in $L^p(\Omega)$ as soon as $f^\pm \in L^p(\partial\Omega)$ with $p \leq 2$. Yet, the problem now is that our domain is an annulus and so, we cannot use the results of \cite{Dweik} to obtain $L^p$ summability on the transport density $\sigma$ (or equivalently, $W^{1,p}$ regularity for the solution $u$ of \eqref{shifted BV least gradient problem original}). So, we want to study the $L^p$ summability of the transport density in the case where the domain is an annulus. First of all, let us assume the following:
$$\mbox{(H5)}\,\,\,\,\,\exists\,  c>0\,\,\,\mbox{s.t.}\,\,\,\forall\,\,x \in \chi_i^-\,(\mbox{resp.}\,\,\Gamma_i^-),\,y \in \chi_i^+\,(\mbox{resp.}\,\,\Gamma_i^+),\,\mbox{we have}\,\,(y-x) \cdot \nu(x) \geq c,$$
where $\nu(x)$ is the outward normal to $\partial\Omega_-$ at $x$. Under (H5), we will show that if $F_i^+$ is a flat part (i.e. $g_+$ is constant on $F_i^+$), then the following statement holds:  
$$ g \in W^{1,p}(\partial\Omega) \Rightarrow u \in W^{1,p} (\Omega),\,\,\mbox{for every}\,\,\,p \in [1,\infty].$$ 
On the other hand, the $W^{1,p}$ estimates, for $p >2$, on the solution $u$ of \eqref{shifted BV least gradient problem} fail to be true as soon as the part $F_i^+$ is not flat. In this case, one can only prove (as in \cite{Dweik}) the following:
$$ g \in W^{1,p}(\partial\Omega) \Rightarrow u \in W^{1,p} (\Omega),\,\,\mbox{for every}\,\,\,p \leq 2.$$ 
In all that follows, $\Omega$ will be an annulus in the sense of Definition \ref{dfn:annulus}. Assume there exists a unique optimal transport plan $\gamma$ between $f^+$ and $f^-$ (for instance the one given by Theorem \ref{thm:existence}). In order to study the $L^p$ summability of the transport density $\sigma$ in the case where the domain $\Omega$ is an annulus, we will use a similar technique as in the proof of \cite[Proposition 3.1]{Dweik}. The main result is the following Theorem:
\begin{theorem}
Under (H5) and the assumption that $F_i^+$ is a flat part for each $i$, the transport density $\sigma$ belongs to $L^p(\Omega)$ as soon as $f \in L^p(\partial\Omega)$, for all $p \in [1,\infty]$. Moreover, if there is some $i$ such that $F_i^+$ is not a flat part, then the same result holds for every $p \leq 2$ as soon as the exterior domain $\Omega_+$ is uniformly convex.  
\end{theorem}
\begin{proof}
First, let us suppose that, for each $i$, $F_i^+$ is a flat part. Now, assume that the target measure $f^-$ is atomic with $(x_{i,j}^\pm)_{\{1 \leq j \leq n\}}$ being its atoms where, for every $i$, $x_{i,j}^+ \in \Gamma_i^+$ and $x_{i,j}^- \in \chi_i^-$, for all $j \in \{1,...,n\}$. Let us call by $\Omega_{i,j}^+$ the set of points of the form $(1-t)x + t x_{i,j}^+$ with $x \in \Gamma_i^-$ and $\Omega_{i,j}^-$ the set of points of the form $(1-t)x  + t x_{i,j}^-$ with $x \in \chi_i^+$. We recall that all the sets $\Omega_{i,j}^\pm$ remain in $\Omega$ and they are essentially disjoint. Let us decompose the transport density $\sigma$ into two parts $\sigma=\sigma^+ + \sigma^-$, where $\sigma^+$ and $\sigma^-$ are defined as follows: 
$$<\sigma^+,\phi>:=\int_{\bar{\Omega} \times \bar{\Omega}} \int_0^{\frac{1}{2}} \phi((1-t)x+ty) |x-y| \,\,\mathrm{d}t\,\,\mathrm{d}\gamma (x,y),\,\,\,\,\mbox{for all}\,\,\,\,\phi \in C(\bar{\Omega}),$$
and 
$$<\sigma^-,\phi>:=\int_{\bar{\Omega} \times \bar{\Omega}} \int_{\frac{1}{2}}^1 \phi((1-t)x+ty) |x-y| \,\,\mathrm{d}t\,\,\mathrm{d}\gamma (x,y),\,\,\,\,\mbox{for all}\,\,\,\,\phi \in C(\bar{\Omega}).$$ \\ 
Now, we want to give some $L^p$ estimates for $\sigma^+$. Recalling Proposition \ref{Prop. 4.2}, there is a unique optimal transport map $S$ from $f^+$ to $f^-$ and then, we have
$$<\sigma^+,\phi>:=\int_{\bar{\Omega}} \int_0^{\frac{1}{2}} \phi((1-t)x+tS(x)) |x-S(x)| \,\,\mathrm{d}t\,\mathrm{d}f^+ (x),\,\,\,\,\mbox{for all}\,\,\,\,\phi \in C(\bar{\Omega}).$$ 
Yet, $\Omega=\bigcup_{i,j} \Omega_{i,j}^\pm$ and, for every $x \in \Omega_{i,j}^\pm$, we have $S(x)=x_{i,j}^\pm$. Hence, we find that
$$<\sigma^+,\phi>:=\sum_{i,j}\int_{\Omega_{i,j}^\pm} \int_0^{\frac{1}{2}} \phi((1-t)x+tx_{i,j}^\pm) |x-x_{i,j}^\pm| \,\,\mathrm{d}t\,\mathrm{d}f^+ (x),\,\,\,\,\mbox{for all}\,\,\,\,\phi \in C(\bar{\Omega}).$$
Set $\Gamma_{i,j}^-=\Omega_{i,j}^+ \cap \Gamma_i^-$ and $\chi_{i,j}^+=\Omega_{i,j}^- \cap \chi_i^+$, for all $i,\,j$. One has that, for all $\phi \in C(\bar{\Omega})$, 
$$<\sigma^+,\phi>=\sum_{i,j}\int_{\Gamma_{i,j}^-} \int_0^{\frac{1}{2}} \phi((1-t)x+tx_{i,j}^+) |x-x_{i,j}^+| \mathrm{d}t\mathrm{d}f^+  + \int_{\chi_{i,j}^+} \int_0^{\frac{1}{2}} \phi((1-t)x+tx_{i,j}^-) |x-x_{i,j}^-| \mathrm{d}t\mathrm{d}f^+.$$ 
We define $\sigma_{i,j}^{+\pm}$ as follows:
$$<\sigma_{i,j}^{++},\phi>=\int_{\Gamma_{i,j}^-} \int_0^{\frac{1}{2}} \phi((1-t)x+tx_{i,j}^+) |x-x_{i,j}^+| \,\mathrm{d}t\,\mathrm{d}f^+,\,\,\,\mbox{for all}\,\,\,\,\phi \in C(\bar{\Omega})$$
and
$$<\sigma_{i,j}^{+-},\phi>= \int_{\chi_{i,j}^+} \int_0^{\frac{1}{2}} \phi((1-t)x+tx_{i,j}^-) |x-x_{i,j}^-| \,\mathrm{d}t\,\mathrm{d}f^+,\,\,\,\mbox{for all}\,\,\,\,\phi \in C(\bar{\Omega}).$$ \\ 
In this way, we have 
$$\sigma^+=\sum_{i,j} \sigma_{i,j}^{++} + \sigma_{i,j}^{+-}.$$ 
We will prove that $\sigma_{i,j}^{+\pm}$ is in $L^p(\Omega)$, for all $i,\,j$, which implies by the way that $\sigma^+$ belongs to $L^p(\Omega)$. Fix $i,\,j$ and consider $\sigma_{i,j}^{++}$. Set $y=(1-t)x + t x_{i,j}^+$, for every $x \in \Gamma_{i,j}^-$ and $t \in [0,1/2]$. Then, one has  
$$<\sigma_{i,j}^{+ +},\phi>:= \int_{{\Omega_{i,j}^{+,\frac{1}{2}}}} \phi(y) \frac{|y-x_{i,j}^+|}{1-t} f^+\bigg(\frac{y-tx_{i,j}^+}{1-t}\bigg){J_{i,j}^+(y)}^{-1}\,\mathrm{d}y,\,\,\,\mbox{for all}\,\,\,\,\phi \in C(\bar{\Omega}),$$
where
$$J_{i,j}^+(y):=\det(D_{(t,x)} y)$$
and
$${\Omega_{i,j}^{+,\frac{1}{2}}}=\{(1-t)x + t x_{i,j}^+\,:\,x \in \Gamma_{i,j}^-,\,0 \leq t \leq 1/2\}.$$\\
An easy estimate for $J_{i,j}^+$ gives that
$$J_{i,j}^+(y)=(1-t) (x_{i,j}^+ - x) \cdot \nu(x),$$
where $\nu(x)$ is the outward normal vector to $\partial\Omega_-$ at $x$. As \,$x \in \Gamma_{i,j}^-$ and \,$x_{i,j}^+ \in \Gamma_{i,j}^+$, we have, by (H5), that
$$(x_{i,j}^+ - x) \cdot \nu(x) \geq c.$$
Consequently, we get
$$ \sigma_{i,j}^{++}(y) \leq \frac{C}{1-t} f^+\bigg(\frac{y-tx_{i,j}^+}{1-t}\bigg),\,\,\,\mbox{for a.e.}\,\,\,y \in {\Omega_{i,j}^{+,\frac{1}{2}}}.$$
Then, 
$$ ||\sigma_{i,j}^{++}||_{L^p(\Omega_{i,j}^+)}^p \leq \int_{\Omega_{i,j}^+}\frac{C^{p}}{(1-t)^{p}}\, {f^+\bigg(\frac{y-tx_{i,j}^+}{1-t}\bigg)}^p\,\mathrm{d}y\leq\int_{\Gamma_{i,j}^-}\int_{0}^{\frac{1}{2}}\frac{C^{p}}{(1-t)^{p-1}} {f^+(x)}^p\,\mathrm{d}t\,\mathrm{d}x$$
$$\leq \bigg(\int_{0}^{\frac{1}{2}}\frac{C^{p}}{(1-t)^{p-1}} \,\mathrm{d}t \bigg)\int_{\Gamma_{i,j}^-}{f^+(x)}^p\,\mathrm{d}x.$$
Similarly, we obtain 
$$ ||\sigma_{i,j}^{+-}||_{L^p(\Omega_{i,j}^-)}^p \leq \bigg(\int_{0}^{\frac{1}{2}}\frac{C^{p}}{(1-t)^{p-1}} \,\mathrm{d}t \bigg)\int_{\chi_{i,j}^+}{f^+(x)}^p\,\mathrm{d}x.$$
This implies that
$$||\sigma^+||_{L^p(\Omega)}^p= \sum_{i,j} ||\sigma_{i,j}^{++}||_{L^p(\Omega_{i,j}^+)}^p + ||\sigma_{i,j}^{+-}||_{L^p(\Omega_{i,j}^-)}^p $$
$$\leq \bigg(\int_{0}^{\frac{1}{2}}\frac{C^{p}}{(1-t)^{p-1}} \,\mathrm{d}t \bigg)\sum_{i,j}\bigg(\int_{\chi_{i,j}^+}{f^+(x)}^p\,\mathrm{d}x + \int_{\Gamma_{i,j}^-}{f^+(x)}^p\,\mathrm{d}x\bigg)\leq C^p\int_{\partial\Omega}{f^+(x)}^p\,\mathrm{d}x.$$\\ 
Passing to the limit when $n \to \infty$, we infer that the positive measure $\sigma^+$ between $f^+$ and $f^-$ is in $L^p(\Omega)$ as soon as $f^+ \in L^p(\partial\Omega)$. Moreover, $\sigma^+$ satisfies the following estimate: 
$$||\sigma^+||_{L^p(\Omega)}^p\leq C^p\int_{\partial\Omega}{f^+(x)}^p\,\mathrm{d}x.$$
But now, it is clear that if $f^- \in L^p(\partial\Omega)$, then one can obtain some $L^p$ estimates on $\sigma^-$ using an approximation of $f^+$ by an atomic sequence. Moreover, we get that
$$||\sigma^-||_{L^p(\Omega)}^p\leq C^p\int_{\partial\Omega}{f^-(x)}^p\,\mathrm{d}x.$$ 
Finally, we infer that
$$||\sigma||_{L^p(\Omega)}\leq C ||f||_{L^p(\partial\Omega)},\,\,\mbox{for every}\,\,p \in [1,\infty].$$ \\
On the other hand, if there is some $F_i^+$ which is not flat, then we can decompose the transport density $\sigma$ into three parts: $\sigma^{+-},\,\sigma^{-+}$ and $\sigma^{++}$, where $\sigma^{+-}$ denotes the transport density between $f^+_{|\,\partial\Omega^+ \backslash \cup_i F_i^{++}}$ and $f^-_{|\partial\Omega^-}$, $\sigma^{-+}$ denotes the transport density between $f^+_{|\partial\Omega^-}$ and $f^-_{|\,\partial\Omega^+ \backslash \cup_i F_i^{+-}}$ and $\sigma^{++}$ denotes the transport density between $f^+_{|\,\cup_i F_i^{++}}$ and $f^-_{|\,\cup_i F_i^{+-}}$. We have already seen that the two transport densities $\sigma^{+-}$ and $\sigma^{-+}$ are both in $L^p(\Omega)$ provided that $f \in L^p(\partial\Omega)$. Yet, if $\Omega_+$ is uniformly convex, then from \cite{Dweik} we have that $\sigma^{++} \in L^p(\Omega)$ as soon as $f \in L^p(\partial\Omega)$ with $p \leq 2$. This completes the proof.   $\qedhere$
\end{proof}
Finally, we get the following:
\begin{corollary}
Under the assumption that $F_i^+$ is a flat part for each $i$, the solution $u$ of the BV least gradient problem \eqref{shifted BV least gradient problem} belongs to $W^{1,p}(\Omega)$ as soon as $g \in W^{1,p}(\partial\Omega)$, for all $p \in [1,\infty]$. On the other hand, if there is some $i$ such that $F_i^+$ is not a flat part, then the same result holds for every $p \leq 2$ as soon as \,$\Omega_+$ is uniformly convex.   
\end{corollary}

\section{Conclusions}\label{Sec. 6}

In this Section, we will present a few examples and closing remarks to show both the limits of the approach presented in Sections 3 \& 4 and the possible extensions of these results. In particular, we will see that while assumptions (H1)-(H4) are not optimal, they are close to optimal.

The first example concerns assumption (H1). We required the measure $f$ to be finite in order to use optimal transport techniques, which translates to the assumption $g \in BV(\partial\Omega)$ in the least gradient problem. However, in the setting of the least gradient problem alone we do not have to assume $g_+ \in BV(\partial\Omega_+)$ and the solutions might still exist.

\begin{example}
Let $\Omega = B(0,2) \backslash \overline{B(0,1)}$. Let $h: [1,2] \rightarrow [1,2]$ be an arbitrary continuous function with infinite total variation and such that $h(1) = 1$. Let $u_0 \in C \cap BV(B(0,2))$ be a solution to the least gradient problem on $B(0,2)$ with boundary data
$$ g_0 (x,y) = \twopartdef{1}{y < 1}{h(y)}{y \geq 1}$$
given by \cite[Theorem 3.7]{Sternberg}. Take the boundary data $g$ equal to $g_-(x,y) = y$ and 
$$g_+ (x,y) = \twopartdef{y}{y < 1}{h(y)}{y \geq 1.} $$
Then, even though condition (H1) is violated, the solution to the least gradient problem exists and equals
$$u(x,y) = \twopartdef{y}{y < 1}{u_0(x,y)}{y \geq 1.} $$
\end{example}

The second example concerns assumption (H2). It requires the boundary data on $\partial\Omega_+$ and $\partial\Omega_-$ to have the same number of monotonicity intervals and determines the total variation on these intervals. By Lemma \ref{lem:tvinequality}, we already know that $TV(g_-) \leq TV(g_+)$; let us see what can happen if the inequality is strict.

\begin{example}
Let $\Omega = B(0,2) \backslash \overline{B(0,1)}$ and set boundary data to equal $g_-(x,y) = 0$ and $g_+ (x,y) = y$. Then the solution to the least gradient problem does not exist.
\end{example}

The third example also concerns assumption (H2). It shows that intervals of monotonicity do not have to be separated by flat parts in order for a solution to exist. However, as we can see from Proposition \ref{stw:specialconfiguration}, this requires a very special configuration of the boundary data.

\begin{example}
Let $\Omega = B(0,2) \backslash \overline{B(0,1)}$ and set boundary data to equal $g(x,y) = y$. Then the solution to the least gradient problem exists even though condition (H2) is violated.
\end{example}

The fourth example concerns assumption (H3). It shows that if the intervals of monotonicity of the boundary data $g$ are not visible from one another, then the solution might not exist.

\begin{example}
Let $\Omega = B(0,M) \backslash \overline{B(0,M -\varepsilon)}$, where $M>0$ is a large constant and $\varepsilon>0$ is small enough. Set boundary data to equal 
$$g_-(x,y) = \threepartdef{0}{\,\,x < 0,}{x}{\,\,x \in [0,1],}{1}{\,\,x > 1,}$$
and
$$g_+(x,y) = \threepartdef{0}{\,\,y < 0,}{y}{\,\,y \in [0,1],}{1}{\,\,y > 1.}$$\\
Then, the visibility condition (H3) fails and if $u \in BV(\Omega)$, it is not possible that each connected component of $\{ u \geq t \}$ for $t \in (0,1)$ is a line segment $l \subset \overline{\Omega}$, hence there is no solution to the least gradient problem.
\end{example}

The second remark concerns an anisotropic version of the least gradient problem. Suppose that $\phi$ is a strictly convex norm on $\mathbb{R}^2$ and that consider the anisotropic least gradient problem with respect to $\phi$. As the only connected minimal surfaces with repsect to $\phi$ are line segments, in light of the analysis performed in \cite{Dweik} for strictly convex domains $\Omega$ we still have a one-to-one correspondence between gradients of $BV$ functions and vector-valued measures with zero divergence. This leads to the following:

\begin{remark}
Suppose that $\phi$ is a strictly convex norm on $\mathbb{R}^2$. Then there is a one-to-one correspondence between minimizers of the following problems:
\begin{equation} \label{Beckmann problem aniso}
  \min \bigg\{ \int_{\bar{\Omega}} \phi(R_{- \frac{\pi}{2}} v)\,:\, v \in \mathcal{M}(\overline{\Omega}; \mathbb{R}^2),\,\,\nabla\cdot v = f \bigg\}
\end{equation}
and 
\begin{equation} \label{shifted BV least gradient problem aniso}
   \min\bigg\{ \int_\Omega \phi(D u) \,:\,u \in BV(\Omega),\, \partial_\tau(Tu) = f \bigg\}.
\end{equation}
Moreover, both problems admit solutions under the assumptions (H1)-(H4) from Section 4 with $\phi$ replacing the Euclidean norm in condition (H4), and Theorem 5.1 also remains true in the anisotropic setting.
\end{remark}

Finally, let us note that the analysis undertaken in this paper could also be used to study the least gradient problem on general Lipschitz domains $\Omega \subset \mathbb{R}^2$, regardless of its homotopy type. In the following Remark, we highlight some results in this paper that could be easily retrieved for general Lipschitz domains.

\begin{remark}
Let $\Omega = \Omega^0 \backslash \bigcup_{i = 1}^N \overline{\Omega_i} \subset \mathbb{R}^2$, where $\Omega^0$ is an open bounded set with Lipschitz boundary and $\Omega_i \subset \subset \Omega^0$ are pairwise disjoint open bounded convex sets. Then, the distances between $\Omega^0$ and $\Omega_i$, as well as distances between $\Omega_i$ and $\Omega_j$ are bounded from below, hence we can reproduce the proof of Lemma \ref{lem:innerbv} and show that on every connected component $\partial\Omega_i$ of the boundary, except from the outer component $\partial\Omega^0$, the trace of a least gradient function has bounded variation. Hence, the assumption that $g \in BV(\partial \Omega)$ is sensible, and under this assumption we may prove equivalence between problems \eqref{Beckmann problem} and \eqref{shifted BV least gradient problem}.

However, if the homotopy type of the domain is highly nontrivial, or if $\Omega_i$ is not strictly convex, the admissibility conditions introduced in Section 4 and required for existence and uniqueness of minimizers would become much more complicated; the same applies to the discussion about $W^{1,p}$ regularity of the least gradient functions in Section 5. Therefore, we have restricted our reasoning to an annulus for clarity. 
\end{remark}

{\bf Acknowledgements.} The authors would like to thank Prof. Piotr Rybka for suggesting that optimal transport methods could be well suited for solving the least gradient problem on nonconvex domains. The work of W. Górny was partly supported by the research project no. 2017/27/N/ST1/02418, "Anisotropic least gradient problem", funded by the National Science Centre, Poland.

\bibliographystyle{plain}

\begin{thebibliography}{}

\bibitem{Beckmann}
{\sc M. Beckmann},
A continuous model of transportation,
{\it{Econometrica}} {\bf 20}, 643--660, 1952.

\bibitem{Bombieri}
{\sc E. Bombieri, E. de Giorgi and E. Giusti}, Minimal cones and the {Bernstein} problem, {\it Invent. Math.} {\bf 7}, 243--268, 1969.v

 \bibitem{DePas1}{\sc L. De Pascale, L. C. Evans and A. Pratelli,}  \newblock integral estimates for transport densities, 
    \newblock {\it Bull. of the London Math. Soc.}. 36, n. 3,pp. 383-395, 2004.
    
     \bibitem{DePas2}{\sc L. De Pascale and A. Pratelli,}
   \newblock Regularity properties for Monge Transport Density and for Solutions of some Shape Optimization Problem,
   \newblock {\it Calc. Var. Par. Diff. Eq}. 14, n. 3, pp.249-274, 2002.
 
  \bibitem{DePas3}{\sc  L. De Pascale and A. Pratelli, }
  \newblock  Sharp summability for Monge Transport density via  Interpolation,
   \newblock{\it ESAIM Control Optim. Calc. Var.} 10, n. 4, pp. 549-552, 2004.
\bibitem{Dweik}
{\sc S. Dweik and F. Santambrogio},
$L^p$ bounds for boundary-to-boundary transport densities, and $W^{1,p}$ bounds for the BV least gradient problem in 2D,
{\it{Calc. Var. Partial Differential Equations}} {\bf 58}, no. 1, 31, 2019.

  \bibitem{DweSan}{\sc S. Dweik and F. Santambrogio,}
\newblock{Summability estimates on transport densities with Dirichlet regions on the boundary via symmetrization techniques,}
\newblock{\it preprint arXiv:1606.00705, 2016.}

\bibitem{EvansGariepy}
{\sc L.C. Evans and R.F. Gariepy}, {\it Measure theory and fine properties of functions}, CRC Press, Boca Raton, 1992.


\bibitem{Giusti}
{\sc E. Giusti}, {\it Minimal surfaces and functions of bounded variation}, Birkh\"{a}user, Basel, 1984.

\bibitem{Gorny}  {\sc W. G\'orny}, Planar least gradient problem: existence, regularity and  anisotropic case,
{\it{Calc. Var. Partial Differential Equations}} {\bf 57}, no. 4, 98, 2018.

\bibitem{GornyGen}
{\sc W. G\'orny}, Existence of minimisers in the least gradient problem for general boundary data, {\it Indiana Univ. Math J.}, to appear.

\bibitem{GornyNonstconv}
{\sc W. G\'orny}, Least gradient problem with respect to a non-strictly convex norm, {\it arXiv:1806.01921}, 2018.

\bibitem{Gorny0}
{\sc W. G\'orny, P. Rybka and A. Sabra},
Special cases of the planar least gradient problem, {\it{ Nonlinear Anal.}} {\bf 151}, 66-–95, 2017.

\bibitem{Kanto}
{\sc L. Kantorovich},
On the transfer of masses,
{\it{Dokl. Acad. Nauk. USSR}}, (37), 7-8, 1942.

\bibitem{Mazon}
{\sc J.M. Maz\'on, J.D. Rossi and S. Segura de L\'eon},  Functions of  least gradient  and 1-harmonic
functions, {\it{Indiana Univ. Math. J.}} {\bf 63}, no. 4, 1067--1084, 2014.

\bibitem{Monge}
{\sc G. Monge},
M\'emoire sur la th\'eorie des d\'eblais et des remblais,
{\it{Histoire de l’Acad\'emie Royale des Sciences
de Paris}}
(1781), 666–704.

\bibitem{Moradifam}{\sc{A. Moradifam}}, Existence and   structure of minimizers of least gradient problems,
{\it{Indiana Univ. Math. J.}} {\bf 67}, no. 3, 1025--1037, 2018.

\bibitem{Sabra}{\sc P. Rybka and A. Sabra},
The planar Least Gradient problem in convex domains, the case of continuous datum, {preprint}, 2019.

 \bibitem{San09}{\sc F. Santambrogio, }
\newblock  Absolute continuity and summability of transport densities: simpler proofs and new estimates,
\newblock {\it Calc. Var. Par. Diff. Eq.} (2009) 36: 343-354.

\bibitem{San2} {\sc F. Santambrogio}, {\it Optimal Transport for Applied Mathematicians}, in {\it Progress in Nonlinear Differential Equations and Their Applications}  87, Birkh\"auser, Basel, 2015.

\bibitem{Spradlin}{\sc G. S. Spradlin and A. Tamasan}, Not all traces on the circle come from functions of least gradient
in the disk,
{\it{Indiana Univ. Math. J.}}
{\bf 63}, no. 6, 1819--1837, 2014.

\bibitem{Sternberg}{\sc P. Sternberg, G. Williams and W.P. Ziemer}, Existence, uniqueness, and regularity for functions of least gradient,
{\it{J. Reine Angew. Math.}},
{\bf 430}, 35--60, 1992. 

\bibitem{Ziemer}{\sc P. Sternberg and W.P. Ziemer}, The {Dirichlet} problem for functions of least gradient, in
{\it{Degenerate diffusions. {The} {IMA} {Volumes} in {Mathematics} and its {Applications}}}, 197--214, 1993. 

\bibitem{Villani}
{\sc C. Villani},
{\it{Topics in Optimal Transportation}},
Graduate Studies in Mathematics. Vol. 58, 2003.
\end{thebibliography}

\end{document}